\documentclass[12pt]{article}
\usepackage{amsmath}
\usepackage{amsthm}
\usepackage{amssymb}
\usepackage{verbatim}
\usepackage{geometry}
\usepackage{graphicx}
\usepackage{color}
\usepackage{mathtools}
\usepackage{nicefrac}
\usepackage{hyperref}

\hypersetup{
    colorlinks = true,
    urlcolor = blue,
    linkcolor = blue,
    citecolor = blue
}

\numberwithin{equation}{section}

\usepackage[utf8]{inputenc}
\usepackage[english]{babel}

\newcommand{\ep}{{\epsilon}}

\newtheorem{thm}{Theorem}[section]
\newtheorem*{thm*}{Theorem}
\newtheorem{lem}[thm]{Lemma}
\newtheorem{prop}[thm]{Proposition}
\newtheorem{cor}[thm]{Corollary}
\newtheorem{rem}[thm]{Remark}

\theoremstyle{definition}

\begin{document}
\title{Normal approximation for functions of \\ hidden Markov models}

\author{Christian Houdr\'e\thanks{School of Mathematics, Georgia Institute of Technology, Atlanta, Georgia 30332-0160, USA. \newline  Email: 
\href{mailto:houdre@math.gatech.edu}{houdre@math.gatech.edu}.  Research supported in part by the grant \# 524678 from the Simons Foundation.} \and George Kerchev\thanks{Universit\'e du Luxembourg, Unit\'e de Recherche en Math\'ematiques, Maison du Nombre, 6 Avenue de la Fonte, L-4364 Esch-sur-Alzette, Grand Duch\'e du Luxembourg. \newline Email: \href{mailto:gkerchev@gmail.com}{gkerchev@gmail.com}. Research partially supported by  TRIAD NSF grant (award 1740776) and the FNR grant APOGee at Luxembourg University (R-AGR-3585-10-C). }}

\maketitle

\abstract{  The generalized perturbative approach is an all purpose variant of Stein's method used to obtain rates of normal approximation. Originally developed for functions of independent random variables this method is here extended to functions of the realization of a hidden Markov model. In this dependent setting, rates of convergence are provided in some applications, leading, in each instance, to an extra log-factor vis a vis the rate in the independent  case.}

\bigskip


\medskip


\bigskip

\section{Introduction} Let $X = (X_1, \ldots, X_n)$ be a random vector with coordinates in a Polish space $E$ and let $f: E^n \to \mathbb{R}$ be a measurable function such that $f(X)$ is square integrable. For a large class of such functions $f$ it is expected that as $n$ grows without bound, $f(X)$ behaves like a normal random variable. To quantify such estimates one is interested in bounding the distance between $f(X)$ and $\mathcal{N} \sim N (m_f, \sigma_f^2)$ where $m_f = \mathbb{E}[f(X)]$ and $\sigma_f^2 = Var (f(X))$. Two  such distances of interest are the Kolmogorov distance
\begin{align}
\notag d_K(f(X), \mathcal{N}) \coloneqq \sup_{t \in \mathbb{R}} |\mathbb{P} (f(X) \leq t) - \mathbb{P}(\mathcal{N} \leq t)|, 
\end{align}
and the Wasserstein distance
\begin{align}
\notag d_W (f(X), \mathcal{N}) \coloneqq \sup_{h \in Lip(1)} | \mathbb{E}[ h(f(X)) ] - \mathbb{E}[h(\mathcal{N})] |,
\end{align}
where this last supremum is taken over real valued functions $h$ such that $|h(x) - h(y)| \leq |x - y|$, for all $x, y \in \mathbb{R}$.

When  the components of $X$ are  independent random variables upper bounds on $d_W(f(X), \mathcal{N})$ were first obtained in~\cite{C1} and  these were extended to $d_K(f(X), \mathcal{N})$ in~\cite{LRP}. Both results rely on a class of difference operators that will be described in Section~\ref{s:decomp}. 

Very few results address the (weakly) dependent case,  and in the present work we provide estimates on $d_K(f(X), \mathcal{N})$ and $d_W(f(X), \mathcal{N})$ when $X$ is generated by a hidden Markov model.  Such a model is of interest from its many applications in fields such as computational biology and speech recognition, see, e.g.,~\cite{DEKM}.  Recall that a hidden Markov model $(Z, X)$ consists of a Markov chain $Z = (Z_1, \ldots, Z_n)$ which emits the observed variables $X = (X_1, \ldots, X_n)$. The possible states in $Z$ are each associated with a distribution on the values of $X$. In other words the observation $X$ is a mixture model where the choice of the mixture component for each observation depends on the component of the previous observation. The mixture components are given by the sequence $Z$.  Note also that given $Z$, $X$ is a Markov chain.

To briefly describe the content of the paper, Section~\ref{s:decomp} contains a short overview of results on normal approximation in the independent setting and introduces a simple transformation  involving iid random variables allowing to adapt these estimates to the hidden Markov model. By exploiting the structure of this construction, we develop further quantitative bounds in Section~\ref{s:bounds}, for the special case when $f$ is a Lipschitz function.   Finally,  applications to variants of the ones analyzed in~\cite{C1} and ~\cite{LRP}, are developed in Section~\ref{s:applications}, leading to an extra log-factor in the various rates obtained there.

\section{Normal approximation via Stein's method}\label{s:decomp}

Let $W \coloneqq f(X)$. Originally  in~\cite{C1}, and then in~\cite{LRP}, various bounds on the distance between $W$ and the normal distribution are obtained through a variant of  Stein's method. As is well known, Stein's method is a way to obtain normal approximation based on the observation that the standard normal distribution $\mathcal{N}$ is the only, centered and unit variance, distribution that satisfies 
\begin{align}
\notag \mathbb{E}[g'(\mathcal{N})] = \mathbb{E}[ \mathcal{N}g(\mathcal{N})],
\end{align}

\noindent for all absolutely continuous $g$ with a.e.~derivative $g'$ such that $\mathbb{E}|g'(\mathcal{N})| < \infty$,~\cite{CGS}, and for the random variable $W$, $|\mathbb{E}[ W g(W) - g'(W)]|$ can be thought of as a distance measuring the proximity of  $W$ to $\mathcal{N}$. In particular, for the Kolmogorov distance, the solutions $g_t$ to the differential equation
\begin{align}
\notag \mathbb{P} (W \leq t) - \mathbb{P}(\mathcal{N} \leq t) = g_t'(W) - W g_t(W),
\end{align}
\noindent are absolutely continuous with a.e.~derivative such that $\mathbb{E}| g_t'(\mathcal{N})| < \infty$,~\cite{CGS}. Then,
\begin{align}
\label{eq:stein} d_K(W, \mathcal{N}) = \sup_{t \in \mathbb{R}}|\mathbb{E}[ g_t'(W) - W g_t(W)]|.
\end{align}

\noindent Further properties of the solutions $g_t$ (see~\cite{LRP}) allow for upper bounds on $\mathbb{E}[ g_t'(W) - W g_t(W)]$ using difference operators associated with $W$ introduced in~\cite{C1}. This is coined as the \emph{generalized perturbative approach} in~\cite{C}, and it is described next. First, we recall  the perturbations used to bound the right-hand side of~\eqref{eq:stein} in~\cite{C1} and~\cite{LRP}.  Let $X' = (X_1', \ldots, X_n')$ be an independent copy of $X$ and let $W' = f(X')$. Then $(W, W')$ is an exchangeable pair since it has the same joint distribution as $(W', W)$. A perturbation  $W^A = f^A(X) \coloneqq f(X^A)$ of  $W$ is defined through the change $X^A$ of  $X$ as follows:
\begin{align}
\notag  X_i^{A} = 
\left\{
	\begin{array}{ll}
		X'_i  & \mbox{if } i \in A, \\
		X_i  & \mbox{if } i \notin A.
	\end{array}
\right.
\end{align}

\noindent for any $A \subseteq [n] \coloneqq \{1, \ldots, n\}$, including $A = \emptyset$. With these definitions, still following~\cite{C1}, difference operators are  defined for any $\emptyset \subseteq A \subseteq [n]$ and $i \notin A$, as:
\begin{align}
\notag  \Delta_i f^A = f(X^A) - f(X^{A \cup \{i\}}). 
\end{align}

\noindent Moreover, set
\begin{align}
\notag T_A(f) \coloneqq \sum_{j \notin A} \Delta_j f(X) \Delta_j f(X^A), \\
\notag T_A'(f) \coloneqq \sum_{j \notin A} \Delta_j f(X) |\Delta_j f(X^A)|,
\end{align}
\noindent and for $k_{n, A} = 1 / \binom{n}{|A|} (n - |A|)$, set
\begin{align}
\notag T_n(f) \coloneqq \sum_{\emptyset \subseteq A \subsetneq [n] } k_{n, A} T_A(f), \\
\notag T'_n(f) \coloneqq \sum_{\emptyset \subseteq A \subsetneq [n]} k_{n, A} T_A'(f).
\end{align}

\noindent Now for $W = f(X_1, \ldots, X_n)$ such that $\mathbb{E} [W ]= 0$, $0 < \sigma^2 = \mathbb{E} [W^2] < \infty$, and assuming all the expectations below are finite, the following result is proved, for $d_W$, in~\cite[Theorem 2.2]{C1}:
\begin{align}
\label{eq:wass_ind} d_W(\sigma^{-1} W, \mathcal{N}) \leq \frac{1}{\sigma^2} \sqrt{Var ( \mathbb{E}[ T_n(f) |X ] ) } + \frac{1}{2 \sigma^3} \sum_{j = 1}^n \mathbb{E} | \Delta_j f(X)|^3,
\end{align}

\noindent while, for $d_K$,~\cite[Theorem 4.2]{LRP} yields:
\begin{align}
\notag d_K(\sigma^{-1} W,  \mathcal{N}) \leq & \frac{1}{\sigma^2} \sqrt{Var( \mathbb{E} [ T_n(f) | X])} +  \frac{1}{\sigma^2} \sqrt{Var( \mathbb{E} [ T_n'(f) | X])}   \\
\label{eq:kol_ind} & + \frac{1}{4 \sigma^3} \sum_{j = 1}^n \sqrt{\mathbb{E} | \Delta_j f|^6} + \frac{\sqrt{2 \pi} } {16 \sigma^3} \sum_{j = 1}^n \mathbb{E} |\Delta_j f(X)|^3,
\end{align}

\noindent where in both cases $\mathcal{N}$ is now a standard normal random variable.


\noindent Our main abstract result generalize~\eqref{eq:wass_ind} and~\eqref{eq:kol_ind} to the case when $X$ is generated by a hidden Markov model. It is as follows:

\begin{prop}\label{thm:wass_hmm} Let $(Z, X)$ be a hidden Markov model with $Z$ an aperiodic time homogeneous and irreducible Markov chain with finite state  space $\mathcal{S}$, and $X$ taking values in a non-empty finite $\mathcal{A}$. Let $W \coloneqq f(X_1, \ldots, X_n)$ with $\mathbb{E}[W] = 0$ and $0 < \sigma^2 = \mathbb{E}[W^2]  < \infty$. Then, there exist a finite  sequence of independent random variables $R = (R_0, R_1, \ldots,  R_{|\mathcal{S}|(n-1)})$, with $R_i$ taking values in $\mathcal{S} \times \mathcal{A}$, for $i = 0, \ldots, |S|(n-1)$, and a measurable function $h: (\mathcal{S} \times \mathcal{A})^{|S|(n-1)+1} \longrightarrow \mathbb{R}$ such that $h(R_0, \ldots, R_{|\mathcal{S}| (n-1)} )$ and $f(X_1, \ldots, X_n)$ are identically distributed. Therefore:
\begin{align}
\label{eq:wass} d_W(\sigma^{-1} W, \mathcal{N} )  \leq \frac{1}{\sigma^2} \sqrt{Var( \mathbb{E}[T_{|R|}(h) |R ]  )} + \frac{1}{2 \sigma^3} \sum_{i = 0}^{|\mathcal{S}|(n-1)} \mathbb{E}|  \Delta_i h(R)|^3.
\end{align}

\noindent  and 
\begin{align}
\notag d_K (\sigma^{-1} W, \mathcal{N})  \leq & \frac{1}{\sigma^2} \sqrt{Var( \mathbb{E} [ T_{|R|}(h) | R])} +  \frac{1}{\sigma^2} \sqrt{Var( \mathbb{E} [ T_{|R|}'(h) | R])}   \\
\label{eq:kol} & + \frac{1}{4 \sigma^3} \sum_{j = 0}^{|R| - 1} \sqrt{\mathbb{E} | \Delta_j h(R)|^6} + \frac{\sqrt{2 \pi} } {16 \sigma^3} \sum_{j = 0}^{|R| - 1} \mathbb{E} |\Delta_j h(R)|^3.
\end{align}
\end{prop}

\noindent At a first glance the above results might appear to be simple corollaries to~\eqref{eq:wass_ind} and ~\eqref{eq:kol_ind}. Indeed, as well known,  every Markov chain (in a Polish space) admits a representation  via iid   random variables $U_1, \ldots, U_n$, uniformly distributed on $(0, 1)$ and the inverse distribution function. Therefore,  $f(X_1, \ldots, X_n) \overset{d}{=} h(U_1, \ldots, U_n)$, for some function $h$, where, as usual, $\overset{d}{=}$ indicates equality in distribution. However, providing quantitative estimates for $\mathbb{E} |\Delta_j h(U_1, \ldots, U_n)|$  via $f$ seems to be out of reach, since passing from $f$ to $h$ involves the ``unknown" inverse distribution function. For this reason, we develop, for our analysis, a more amenable, although more restrictive, choice of iid random variables  described intuitively in the next paragraph and then again in greater details in Section~\ref{s:construction}.

\noindent Consider $R = (R_0, \ldots, R_{|\mathcal{S}|(n-1)})$ as stacks of independent random variables on the $|\mathcal{S}|$ possible states of the hidden chain that determine the next step in the process, with $R_0$ specifying the initial state. Each $R_i$ takes values in $\mathcal{S} \times \mathcal{A}$ and is distributed according to the transition probability from the present hidden state.  Then, one has $f(X_1, \ldots , X_n) \overset{d}{=} h (R_0, \ldots , R_{|\mathcal{S}|(n-1)})$, for $h =  f \circ \gamma$,  where the function $\gamma$ translates between $R$ and $X$, and where $\overset{d}{=}$ indicates equality in distribution. This construction is carried out in more details in the next section. Further note that  when $(X_i)_{i \geq 1}$ is a sequence of independent random variables, the hidden chain in the model consists of a single state and then the function $\gamma$ is the identity function.

\begin{rem} As observed in~\cite{CSZ}, the terms involving $\Delta_i h(R)$ in~\eqref{eq:wass} and~\eqref{eq:kol} can be removed, leaving only the variance terms. Here is a different way to establish this fact. Indeed, recall that the expressions on the right-hand side of~\eqref{eq:wass} and~\eqref{eq:kol} are bounds on terms of the form $\mathbb{E}| g_t'(W) - g_t'(W) T| + |\mathbb{E} [g_t(W) W  - g_t'(W) T] |$, where $|g_t'| \leq 1$ and $|g_t(W)W - g_t'(W)| = |\mathbf{1}_{W \leq t} - \mathbb{P}(\mathcal{N} \leq t)| \leq 1$ (see~\cite{LRP} and~\cite{C}). First, note that
\begin{align}
\notag |g_t'(W) - g_t'(W)T | \geq |g_t'(W) T | - 1, 
\end{align} 
\noindent and
\begin{align}
\notag 1 \geq |g_t(W)W - g_t'(W)| \geq |g_t(W) W| - 1.
\end{align}
\noindent Then, by the triangle inequality and the above,
\begin{align}
\notag |g_t(W) W  - g_t'(W) T| \leq |g_t(W) W| + |g_t'(W) T| \leq  |g_t'(W) - g_t'(W)T | + 3.
\end{align}
\noindent Therefore, if $  \mathbb{E} | g_t'(W) - g_t'(W) T|  / \sigma^2 \to 0$, then 
\begin{align}
\notag | \mathbb{E}[ g_t(W)W - g_t'(W)T]| /   \sigma^3 \leq C   \mathbb{E} | g_t'(W) - g_t'(W) T|  / \sigma^2, 
\end{align} 
\noindent for some constant $C > 0$ that does not depend on $n$. Therefore, the asymptotic behavior of the bounds in~\eqref{eq:wass} and~\eqref{eq:kol} is given by the terms corresponding to $\mathbb{E} | g_t'(W) - g_t'(W)T|$, i.e., the terms involving the variance. This modification of the method is also valid in our framework and would ``improve" our results. However, this has no really significant  incidence on the rates obtained in our applications in Section~\ref{s:applications}, and so this will not be pursued here any further.

\end{rem}

\subsection{Construction of $R$}\label{s:construction}

Let $(Z, X)$ be a hidden Markov model with $Z$ an aperiodic time homogeneous and irreducible Markov chain on a finite state space $\mathcal{S}$, and $X$ taking values in an alphabet $\mathcal{A}$. Let $P$ be transition matrix of the hidden chain and let  $Q$ be the $|\mathcal{S}| \times |\mathcal{A}|$ probability matrix for the observations, i.e., $Q_{ij}$ is the probability of seeing output $j$ if the latent chain is in state $i$. Let the initial distribution of the hidden chain be $\mu$. Then 
\begin{align}
\notag & \mathbb{P} \bigg( (Z_1, \ldots, Z_n; X_1, \ldots, X_n) = (z_1, \ldots, z_n; x_1, \ldots, x_n) \bigg) \\
\notag   & \quad \quad  \quad   =  \mu(z_1) Q_{z_1, x_1} P_{z_1, z_2} \ldots P_{z_{n-1}, z_n} Q_{z_n, x_n}. 
\end{align}

\noindent Next we introduce a sequence of independent random variables $R_0, \ldots, R_{|\mathcal{S}|(n-1)}$ taking values in $\mathcal{S} \times \mathcal{A}$ and a function $\gamma$ such that $\gamma(R_0, \ldots, R_{|\mathcal{S}| (n-1) }) = (Z_1, \ldots, Z_n;$ $ X_1, \ldots , X_n)$.  For any $s, s' \in \mathcal{S}$, $x \in \mathcal{A}$ and $i \in \{0, \ldots, n-1\}$, let
\begin{align}
\notag  \mathbb{P}\big(R_0 = (s, x) \big)   & = \mu(s) Q_{s, x}, \\
\notag  \mathbb{P} \big(R_{i |S| + s'} = (s, x) \big) &  = P_{s' , s} Q_{s, x}. 
\end{align}

\noindent The random variables $R_i$ are well defined since $\sum_x Q_{s, x} = 1,$  for any $s \in \mathcal{S}$, and $\sum_s P_{s', s} = \sum_s \mu(s) = 1,$ for any $s' \in \mathcal{S}$. One can think of the variables $R_i$ as a set of instructions indicating where the hidden Markov model goes next. The function $\gamma$ reconstructs the realization $(Z_i, X_i)_{i \geq 1}$ sequentially from the sequence $(R_i)_{i \geq 0}$. In particular, $\gamma$ captures the following relations
\begin{align}
\notag  (Z_1, X_1) & = R_0, \\
\notag  (Z_{i+1}, X_{i+1}) & = R_{i |S| + s} \mbox{ , if } Z_i = s \mbox { for } i \geq 1.
\end{align}

\noindent One can also think of the sequence $(R_i)_{i \geq 0}$ as $|\mathcal{S}|$ stacks of random variables on the $\mathcal{S}$ possible states of the latent Markov chain, and the values being rules for the next step in the model. Note that only one variable on the $i$th level of the stack will be used to determine the $(i+1)$-st hidden and observed pair. Furthermore, the distribution of the random variables $R_i$, for $i \geq 1$ encodes the transition and output probabilities in the $P$ and $Q$ matrices of the original model.

\noindent Thus one can write $f(X_1, \ldots, X_n) = h(R_0, \ldots , R_{|\mathcal{S}| (n-1) })$, for $h \coloneqq f \circ \gamma$, where the function $\gamma$ does the translation from $(R_i)_{i \geq 0}$ to $(Z_i, X_i)_{i \geq 1}$ as described above.

\noindent Let $R' = (R_0', \ldots, R_{|\mathcal{S}| (n-1)}')$ be an independent copy of $R$. Let $A \subseteq \{0, 1, \ldots, |S|(n-1)\}$ and let the change $R^A$ of $R$ be defined as follows
\begin{align}
\label{eq:rder_def} R_i^A = \left\{ \begin{array}{rl} R_i' & \mbox { if } i \in A \\ R_i & \mbox{ if } i \notin A, \end{array}\right.
\end{align}

\noindent where, as before, when $A = \{j\}$ we write $R^j$ instead of $R^{\{j\}}$.

\noindent Recall that the ``discrete derivative" of $h$ with a perturbation $A$ is 
\begin{align}
\notag \Delta_i h^A = h(R^A) - h (R^{A \cup \{i\}}). 
\end{align}

\noindent Then~\eqref{eq:wass} and~\eqref{eq:kol} follow from~\eqref{eq:wass_ind} and~\eqref{eq:kol_ind}, respectively, since when $(Z, X)$ is a hidden Markov model one writes
\begin{align}
\notag W = f(X_1, \ldots, X_n) \overset{d}{=} h (R_0, \ldots, R_{|S|(n-1)}),
\end{align}

\noindent where the sequence $(R_i)_{i \geq 0}$ is a sequence of independent random variables.

\begin{rem}\noindent{ (i)} The idea for using stacks of independent random variables to represent a hidden Markov model is somehow reminiscent of  Wilson's cycle popping algorithm for generating a random directed spanning tree, see~\cite{W}. The algorithm has also been related to loop-erased random walks in~\cite{GP}. 

\noindent{(ii)} If $S$ consists of a single state, making the hidden chain redundant, there is a single stack of instructions. This corresponds to the independent setting of~\cite{C1} and~\cite{LRP}, and then $\gamma$ is just  the identity function.

\noindent{(iii)} The same approach of using instructions is also applicable when $\mathcal{A}$ and $\mathcal{S}$ are countable. The $Q_{s,x}$ no longer form a finite matrix but the same definition holds as long as $\sum_{x \in \mathcal{A}} Q_{s,x} = 1$, for all $s \in \mathcal{S}$. We need countably infinite independent instructions to encode $(Z_i, X_i)_{1 \leq i \leq n}$. In particular, let $R_0$ and $(R_{i, s})_{1 \leq i \leq n, s \in \mathcal{S}}$ be such that
\begin{align}
\notag \mathbb{P}(R_0 = (s,x) ) = \mu(s) Q_{s,x}, \\
\notag \mathbb{P}(R_{i,s'} = (s, x)) = P_{s', s} Q_{s,x}.
\end{align} 
\noindent Then the function $\gamma$ reconstructs $(Z_i, X_i)_{1 \leq i \leq n}$ from $R_0$ and $(R_{i,s})_{1 \leq i \leq n, s \in \mathcal{S}}$ via
\begin{align}
\notag (Z_1, X_1) & = R_0, \\
\notag (Z_{i+1}, X_{i+1} ) & = R_{i, s}, \mbox{ if }  Z_i = s \mbox { for } i \geq 1.
\end{align} 

\end{rem}

\section{Further quantitative bounds}\label{s:bounds}

In the present section several bounds on the quantities appearing on the right-hand side of~\eqref{eq:wass} and~\eqref{eq:kol} are presented, under some standard assumption on the underlying hidden Markov model. Furthermore, assuming a Lipschitz property for the function $f$ in $W \coloneqq f(X)$, it is shown that up to a log factor, $Var ( f(X)) $ is linearly upper-bounded in the size of $X$.

\subsection{Bounds on $\Delta_i h$}

Again, let the latent chain in the hidden Markov model be irreducible and aperiodic, with finite state space $\mathcal{S}$, and assume that it is started at the stationary distribution. Then there exist $K \geq 1$, and $\ep \in (0,1)$, such that
\begin{align}
\notag \mathbb{P} (Z_n = s, Z_{n + K} = s') \geq \ep,
\end{align}
\noindent and thus,
\begin{align}
\label{eq:ep_K}  \mathbb{P}(Z_{n + K} = s') \geq \ep, \quad \mathbb{P}(Z_{n+K} = s' | Z_n = s) \geq \ep, 
\end{align}
\noindent for all $n \geq 1$ and $s, s' \in \mathcal{S}$. A simple corollary of these facts is the following lemma.

\begin{lem}\label{lem:ind_dec} Let $K \geq 1$ and $\ep \in (0,1)$ be as in~\eqref{eq:ep_K} and $(Z_i)_{i \geq 1}$ be an irreducible and aperiodic Markov chain with finite state space $\mathcal{S}$. Then, 
\begin{align}
\label{eq:ind_dec} \mathbb{P}\left(Z_{j+K} \neq s_1, Z_{j+2K} \neq s_2, \ldots, Z_{j+tK} \neq s_t\right)  \leq (1 - \ep)^t,
\end{align}
\noindent for any $t \geq 1$, $j \geq 1$ and $(s_1, \ldots, s_t) \in \mathcal{S}^t$.
\end{lem}

\begin{proof} We show~\eqref{eq:ind_dec} by induction. The case $t = 1$ follows from~\eqref{eq:ep_K}. Next, for $(s_1, \ldots, s_{t+1}) \in \mathcal{S}^{t+1}$,
\begin{align}
\notag \mathbb{P}& \left(Z_{j+K} \neq s_1, Z_{j+2K} \neq s_2, \ldots, Z_{j+(t+1)K} \neq s_{t+1}\right) \\
\notag = & \quad \sum_{s_1' \neq s_1, \ldots, s_{n+1}' \neq s_{t+1} } \mathbb{P} (Z_{j+K}  = s_1', \ldots, Z_{n+1} = s_{t+1}') \\
\notag = & \sum_{s_1' \neq s_1, \ldots, s_{n+1}' \neq s_{t+1} } \mathbb{P} (Z_{j+(t+1)K} = s_{t+1}' | Z_{j+K}  = s_1', \ldots, Z_{j + tK} = s_{t}') \\
\notag & \quad \quad \quad \quad\quad\quad \quad\quad \quad \cdot \mathbb{P} ( Z_1 = s_1', \ldots, Z_{j + tK} = s_{t}') \\
\notag = &  \sum_{s_1' \neq s_1, \ldots, s_{n+1}' \neq s_{t+1} } \mathbb{P} (Z_{j+(t+1)K} = s_{t+1}' |  Z_{j + tK} = s_t') \mathbb{P} (  Z_{j+K}  = s_1', \ldots, Z_{j + tK} = s_{t}') \\ 
\notag = &  \sum_{s_1' \neq s_1, \ldots, s_{t}' \neq s_{t} } \mathbb{P} (Z_{j+(t+1)K} \neq s_{t+1} |  Z_{j + tK} = s_{t}') \mathbb{P} (  Z_{j+K}  = s_1', \ldots, Z_{j + tK} = s_{t}') \\ 
\notag \leq & (1 - \ep) \sum_{s_1' \neq s_1, \ldots, s_{t}' \neq s_{n} } \mathbb{P} (  Z_{j+K}  = s_1', \ldots, Z_{j + tK} = s_{t}') \\ 
\notag = & (1 - \ep) \mathbb{P} (Z_{j+K}  \neq s_1, \ldots, Z_{tK} \neq s_t) \\
\notag \leq & (1 - \ep)^{t+1}, 
\end{align}
\noindent where we have used the Markov property,~\eqref{eq:ep_K} and finally the induction hypothesis. This suffices for the proof of~\eqref{eq:ind_dec} and thus the proof of the lemma is complete. 
\end{proof}

\noindent The next result provides first a tail inequality from which moments can be estimated.

\begin{prop}\label{prop:der_tail} Let $(Z, X)$ be a hidden Markov model as above and let $K > 0$ and $\ep > 0$ be as in~\eqref{eq:ep_K}. Let $g:\mathcal{A}^n \to \mathbb{R}$ be Lipschitz, i.e., be such that $|g(x) - g(y)| \leq  c\sum_{i =1}^n \mathbf{1}_{x_i \neq y_i}$, for every $x, y \in \mathcal{A}^n$, and where $c > 0$. Let $R = (R_0, \ldots, R_{|S|(n-1)})$ be a vector of independent random variables and $h$ be the function such that
\begin{align}
\notag g(X_1, \ldots, X_n) \overset{d}{=} h(R_0, \ldots,  R_{|S|(n-1)}).
\end{align}
\noindent Then, for $R^i$, as defined in~\eqref{eq:rder_def},
\begin{align}
\label{eq:der_pr} \mathbb{P} (|h(R) - h(R^i)| \geq c x   ) \leq C (1 - \ep)^{x/K},
\end{align}
\noindent for any $x \in \mathbb{N}$, and where $C > 0$ depends on the parameters of the model but neither on $n$ nor on $x$.  Then, for any $r > 0$, 
\begin{align}
\label{eq:der_exp} \mathbb{E}| h(R) - h(R^i)|^r \leq \tilde{C} (\ln n)^r,
\end{align}
\noindent for $n$ large enough and where $\tilde{C} = \tilde{C}(r)$.

\end{prop}

\begin{proof}

The sequence of instructions $R^i$ may give rise to a different realization $(Z', X')$ of the hidden Markov model, as compared to $(Z,X)$ - the one generated by $R$. The two models are not independent. In particular, if instruction  $R_i$ determines $(Z_j, X_j)$ and $R_i'$ determines $(Z_j', X_j')$ then $(Z_k, X_k) = (Z'_k, X_k')$ for $k < j$. Let $s$ be the smallest nonnegative integer (possibly $s = \infty$) such that $Z_{j+s} = Z'_{j+s}$. Then for any $k > j+s$, $(Z_k, X_k) = (Z'_k, X_k')$ as well. Finally, if $k \in \{j, \ldots, j + s - 1\}$, the pairs $(Z_k, X_k)$ and $(Z_k', X_k')$ are independent.   We show next,  that  for  $K \geq 1$ as in~\eqref{eq:ep_K}, and any $t \in \mathbb{N}$,
\begin{align}
\label{eq:s_decay}  \mathbb{P}(s \geq tK) \leq (1 - \ep)^t.
\end{align}

\noindent Indeed, 
\begin{align}
\notag \mathbb{P}(s > tK) \leq & \quad  \mathbb{P} \left(Z_{j + K} \neq Z_{j+K}',Z_{j + 2K} \neq Z_{j+2K}',\ldots,  Z_{j + tK} \neq Z_{j+tK}'\right) \\ 
\notag  = & \quad \sum_{(s_1, \ldots, s_t) \in \mathcal{S}^t} \mathbb{P} \left(Z_{j + K} \neq s_1, Z_{j+K}' = s_1, \ldots,  Z_{j + tK} \neq s_t, Z_{j+tK}' = s_t\right).
\end{align}
\noindent By independence, 
\begin{align}
\notag  \mathbb{P} &  \left(Z_{j + K} \neq s_1, Z_{j+K}' = s_1, \ldots,  Z_{j + tK} \neq s_t, Z_{j+tK}' = s_t\right) \\
\notag = & \quad \mathbb{P}(Z_{j+K}  \neq s_1, \ldots, Z_{tK} \neq s_t) \mathbb{P}(Z_{j+K}' = s_1, \ldots, Z_{tK}' =  s_t),
\end{align}
\noindent and thus by Lemma~\ref{lem:ind_dec}
\begin{align}
\notag \mathbb{P}(s > tK) \leq&  \quad \sum_{(s_1, \ldots, s_t)} ( 1- \ep)^t \mathbb{P}(Z_{j+K}' = s_1, \ldots, Z_{tK}' =  s_t) \\
\notag \leq & \quad (1 - \ep)^t,
\end{align}
\noindent as desired.

\noindent Let $E(t)$ be the event 
\begin{align}
\notag E(t) \coloneqq \{X_{j + K} \neq X_{j + K}', X_{j + 2K} \neq X_{j + 2K}' , \ldots, X_{j + tK} \neq X_{j + tK}'\}.
\end{align}

\noindent Note that $\mathbb{P} ( E(t) ) \leq \mathbb{P}(s \geq tK)  \leq (1 - \ep)^{t}$. In particular, if $|h(R) - h(R^i)| \geq cx$, where $c > 0$ is the Lipschitz constant of $g$, then $s \geq x$, as there are at least $x$ positions $k$ such that $X_k \neq X_k'$. Thus,
\begin{align}
\notag  \mathbb{P} ( |h(R) - h(R^i) | \geq cx) \leq  & \quad \mathbb{P} (E(\lfloor x/K \rfloor)) \\
\label{eq:der_pr_proof} \leq & \quad C (1 -  \ep)^{x/K},
\end{align}

\noindent where $C > 0$ depends on the parameters of the model but not on $x$. This suffices for the proof of~\eqref{eq:der_pr}. Next for~\eqref{eq:der_exp},  let $E_t$ be the event that $|h(R) - h(R^i)| \geq tK$. Then
\begin{align} 
\notag \mathbb{E} |h(R) - h(R^i)|^r = \mathbb{E} |h(R) - h(R^i)|^r \mathbf{1}_{E_t} + \mathbb{E} |h(R) - h(R^i)|^r \mathbf{1}_{E_t^c},
\end{align}

\noindent Recall that $|g(x)| \leq cn$, for all $x \in \mathcal{A}^n$, and then $|h(R) - h(R^i)| \leq 2cn$. Using~\eqref{eq:der_pr_proof}, 
\begin{align} 
\notag \mathbb{E} |h(R) - h(R^i)|^r   & \leq (2cn)^r \mathbb{P}(E_t) + (ctK)^r \mathbb{P}(E_t^c)  \\
\label{eq:pertub} & \leq  (2cn)^r (1 - \ep)^{t} + (ctK)^r.
\end{align}

\noindent  Let $t = -  r \ln n  / (\ln (1 - \ep) ) > 0$. Then,
\begin{align}
\label{eq:pertub_bound} \mathbb{E}|h(R) - h(R^i) |^r \leq    (2c)^r + \left( -  \frac{ c r K }{\ln (1 - \ep) } \right)^r (\ln n)^r.
\end{align}

\noindent The order of the bound is optimal for $t$ such that 
\begin{align}
\label{eq:order_xn} (1 - \ep)^{t } \leq \left( \frac{\ln n}{n}\right)^r,
\end{align}
\noindent or
\begin{align}
\notag t \geq - \frac{r (\ln n - \ln (\ln n)) }{\ln (1 - \ep)},
\end{align}
\noindent it follows that 
\begin{align}
\notag \mathbb{E}|h(R) - h(R^i) |^r \leq   (2c)^r + \left( -  \frac{c  r K }{\ln (1 - \ep) } \right)^r (\ln n - \ln (\ln n))^r,
\end{align}
\noindent and the right-hand side has the same order of growth as~\eqref{eq:pertub_bound}.

\noindent If the growth order of $(1 - \ep)^{t}$ is larger than the one in~\eqref{eq:order_xn}, the bound on the second term in~\eqref{eq:pertub} is of larger order as well.

\end{proof}

\begin{rem} Recall that in the independent setting, there is a single stack, or equivalently the state space of the latent chain consists of a single element. Then for $s$ defined in the first paragraph of the above proof, $\mathbb{P}(s > 1) = 0$. Thus we can take $tK = 2$, and since $\mathbb{P}(E_t ) \leq \mathbb{P}(s \geq tk) = 0$,~\eqref{eq:pertub} becomes
\begin{align} 
\notag \mathbb{E} |h(R) - h(R^i)|^r \leq (2c)^r,
\end{align}
\noindent which recovers the independent case.
\end{rem}

\noindent The proposition just proved leads to upper bounds on the central moments of $g(X)$.

\begin{cor}\label{cor:var} Let $(Z, X)$ be a hidden Markov model as above. Let $g : \mathcal{A}^n \to \mathbb{R}$ be such that $|g(x) - g(y) | \leq c \sum_{i =1 }^n \mathbf{1}_{x_i \neq y_i}$, for all $x, y \in \mathcal{A}^n$ and where $c > 0$. Then, for any $r > 0$, 

\begin{align}
\notag  \mathbb{E}|g(X_1, \ldots, X_n) - \mathbb{E}[g(X_1, \ldots, X_n)]|^r  \leq C n^{r/2} (\ln n)^r, 
\end{align}

\noindent for  $n$ large enough and where $C = C(|\mathcal{S}|, r)$.
\end{cor}

\begin{proof} As in Proposition~\ref{prop:der_tail} let $R = (R_0, \ldots, R_{|\mathcal{S}| (n-1)})$ be a vector of independent random variables and $h$ be a function such that
\begin{align}
\notag g(X_1, \ldots, X_n) = h(R).
\end{align}

\noindent Let $R' = (R_0', \ldots, R_{|\mathcal{S}|(n-1)}')$ be an independent copy of $R$.  Recall the generalization of the Efron-Stein inequality (see~\cite{HM},~\cite{RT}) for $r \geq 2$
\begin{align}
\notag  (\mathbb{E}|h(R) - \mathbb{E}h(R)|^r)^{1/r} \leq \frac{r - 1}{2^{1/r}} \left( \sum_{i = 0}^{|R|-1} (\mathbb{E}| h(R) - h(R^i)|^r )^{2/r}\right)^{1/2},
\end{align}
\noindent and for $ r \in (0,2)$, 
\begin{align}
\notag  (\mathbb{E}|h(R) - \mathbb{E}h(R)|^r)^{1/r} \leq \frac{1}{\sqrt{2}} \left( \sum_{i = 0}^{|R|-1} \mathbb{E}| h(R) - h(R^i)|^2 \right)^{1/2},
\end{align}

\noindent with $R^i$ defined as in Proposition~\ref{prop:der_tail}. 

\noindent By Proposition~\ref{prop:der_tail} there is $C > 0$, such that, for all $r > 0$,
\begin{align}
\notag \mathbb{E}|h(R) - \mathbb{E}h(R)|^r  \leq &  \left( \max \left\{\frac{1}{\sqrt{2}},  \frac{r-1}{2^{1/r}} \right\} \right)^r \left(  ( |S| (n-1) +1 ) C (\ln n)^2 \right)^{r/2} \\
\notag \leq &  C' n^{r/2} (\ln n)^r, 
\end{align}

\noindent where $C'>0$ is a function of $|S|$ and $r$. Finally, note that $g(X_1, \ldots, X_n) \overset{d}{=} h(R)$.
\end{proof}

\begin{rem} Note that the bound on the central moments also follows from using an exponential bounded difference inequality for Markov chains proved by Paulin~\cite{DP}. This holds for the general case when $X$ is a Markov chain (not necessarily time homogeneous), taking values in a Polish space $\Lambda = \Lambda_1 \times \cdots \times \Lambda_n$, with mixing time $\tau_{min}$. Then for any $ t \geq 0$, 
\begin{align}
\notag \mathbb{P}( | f(X) - \mathbb{E} [f(X)] | \geq t) \leq 2 \exp \left( \frac{- 2t^2}{||c^*||^2 \tau_{min}} \right),
\end{align}
\noindent where $f$ is such that
\begin{align}
\notag |f(x) - f(y)| \leq \sum_{i =1}^n c_i \mathbf{1}_{x_i \neq y_i},
\end{align}
\noindent for any $x, y \in \mathbb{R}^n$ and some $c^* = (c_1, \ldots, c_n) \in \mathbb{R}^n$, and where $||c^*||^2 = \sum_{i = 1}^n c_i^2$. 
\end{rem}

\subsection{Bounds on the variance terms in~\eqref{eq:wass} and~\eqref{eq:kol}}

Let $U \coloneqq  \sum_{ \emptyset \subseteq A \subsetneq[|R|]} k_{|R|, A} U_A /2$ for a general family of square-integrable random variables $U_A(R, R')$. From~\cite[Lemma 4.4]{C1}, 
\begin{align}
\notag \sqrt{Var(\mathbb{E}[U|R])} \leq & \frac{1}{2} \sum_{\emptyset \subseteq A \subsetneq[|R|]} \sqrt{Var(\mathbb{E}[ U_A | R])} \\
\notag \leq & \frac{1}{2} \sum_{\emptyset \subseteq  A \subsetneq[|R|]} \sqrt{\mathbb{E}[ Var(U_A | R')]}
\end{align}

\noindent As in~\cite{LRP} this inequality will be used for both $U_A = T_A(h)$ and $U_A = T_A'(h)$. A major difference from the setting in~\cite[Section 5]{LRP} is that the function $h$ is not symmetric, i.e., if $\sigma$ is a permutation of $\{ 0, \ldots, |S|(n-1)\}$, it is not necessarily the case that $h(R_0, \ldots, R_{|S|(n-1)}) = h(R_{\sigma(0)}, \ldots, R_{\sigma(|R|(n-1))})$.  Indeed, each variable in $R$ is associated with a transition at a particular step and from a particular state. Fix $A \subsetneq[|R|]$ and let $\tilde{R}$ be another independent copy of $R$. Introduce the substitution operator 
\begin{align}
\notag \tilde{S}_i(R) = (R_0, \ldots, \tilde{R}_i, \ldots, R_{|R|}).
\end{align}

\noindent Recall that from the Efron-Stein inequality, 
\begin{align}
\notag Var( U_A | R') \leq \frac{1}{2} \sum_{i = 0}^{|R|-1} \mathbb{E}[ ( \tilde{\Delta}_i U_A(R) )^2 | R'], 
\end{align}
\noindent where $\tilde{\Delta}_i U_A(R) = U_A(\tilde{S}_i(R)) - U_A(R)$.

\noindent Then,
\begin{align}
\label{eq:var_UA} \sqrt{Var( \mathbb{E}[ U| R] )} \leq \frac{1}{\sqrt{8}} \sum_{\emptyset \subseteq	 A \subsetneq [|R|]} k_{|R|, A} \sqrt{\sum_{i = 0}^{|R| -1} \mathbb{E}[ \tilde{\Delta}_i U_A]^2}.
\end{align}

\noindent Recall also that $U_A = \sum_{j \notin A} \Delta_j h(R) a (\Delta_j h(X^A))$, where the function $a$ is either the identity, or $a (\cdot) = | \cdot |$. Then
\begin{align}
\notag \sum_{i = 0}^{|R| -1}  \mathbb{E} [ \tilde{\Delta}_i U_A]^2  =   \sum_{i = 0}^{|R| - 1}  \sum_{j, k \notin A} & \mathbb{E} [| \tilde{\Delta}_i ( \Delta_j h(R) a (\Delta_j h(R^A)))| \\
\label{eq:57} & \times |\tilde{\Delta}_i ( \Delta_k h(R) a(\Delta_k h(R^A)) ) |].  
\end{align}

\noindent Fix $0 \leq i \leq |R| - 1$, and note that for $j \notin A$, 
\begin{align}
\notag  & \tilde{\Delta}_i ( \Delta_j h(R) - a( \Delta_j h(R^A)) ) \\
\label{eq:58} & \quad =  \tilde{\Delta}_i (\Delta_j h(R)) a(\Delta_j h(R^A) + \Delta_j h(\tilde{S}_i(R)) \tilde{\Delta}_i(a (\Delta_j h(R^A))).
\end{align}

\noindent Then, using  $|\tilde{\Delta}_i a( \cdot) | \leq |\tilde{\Delta}_i (\cdot)|$, the summands in~\eqref{eq:57} are bounded by
\begin{align}
\label{eq:511} 4 \sup_{Y, Y', Z, Z' } \mathbb{E} |\tilde{\Delta}_i ({\Delta}_j h(Y) ) {\Delta}_j h(Y') \tilde{\Delta}_i ({\Delta}_k h(Z) ) {\Delta}_k h(Z')|,
\end{align}
\noindent where $Y, Y', Z, Z'$ are recombinations of $R, R', \tilde{R}$, i.e., $Y_i \in \{R_i, R_i', \tilde{R}_i\}$, for $i \in [0, |R| -1]$.

\noindent Next, as in~\cite{LRP},  bound each type of summand appearing in~\eqref{eq:57}. 

\noindent If $i = j = k$ and using  $\tilde{\Delta}_i (\Delta_i (\cdot)) = \Delta_i (\cdot)$,~\eqref{eq:511} is bounded by
\begin{align}
\notag 4 \sup_{Y, Y', Z, Z' } \mathbb{E} | \Delta_i h(Y)  \Delta_i h(Y') \Delta_i h(Z)  \Delta_i h(Z')| \leq 4 \mathbb{E} |\Delta_i h(R)|^4.
\end{align}

\noindent If $i \neq j \neq k$, switch $\tilde{R}_i$ and $R_i'$, and $Y$ is still a recombination. Then~\eqref{eq:511} is equal to
\begin{align}
\notag   4  & \sup_{Y, Y', Z, Z' } \mathbb{E}[ {\Delta}_i ({\Delta}_j h(Y) ) {\Delta}_j h(Y') {\Delta}_i ({\Delta}_k h(Z) ) {\Delta}_k h(Z')] \\
\notag & \leq  \quad  4 \sup_{Y, Y', Z, Z'} \mathbb{E}[ \mathbf{1}_{\Delta_{i,j} h(Y) \neq 0} ( | \Delta_j h(Y) | + |\Delta_j h(Y^i)| ) | \Delta_j h(Y') | \\
\notag & \quad \times \mathbf{1}_{\Delta_{i,k} h(Z) \neq 0} ( | \Delta_k h(Z) | + | \Delta_k h(Z^i) | ) |\Delta_k h(Z')|] \\
\label{eq:ijk} & \leq  \quad 16 \sup_{Y, Y', Z, Z'} \mathbb{E} [\mathbf{1}_{\Delta_{i,j} h(Y) \neq 0, \Delta_{j,k} h(Y') \neq 0} |\Delta_j h(Z)|^2 |\Delta_k h(Z')|^2],
\end{align}

\noindent where the last step follows from the Cauchy-Schwarz inequality.

\noindent If $i \neq j = k$,~\eqref{eq:511} is equal to
\begin{align}
\notag  4  & \sup_{Y, Y', Z, Z'}  \mathbb{E} | \tilde{\Delta}_i (\Delta_j (h(Y)) \Delta_j h(Y') \tilde{\Delta}_i (\Delta_j (h(Z)) \Delta_j h(Z')| \\
\notag & =  \quad 4 \sup_{Y, Z}  \mathbb{E} | \tilde{\Delta}_i (\Delta_j (h(Y))^2 \Delta_j h(Z)^2 | \\
\notag & =  \quad 4 \sup_{Y, Z} \mathbb{E} | \Delta_j (\Delta_i (h(Y))^2 \Delta_j h(Z)^2| \\
\label{eq:i} & \leq  \quad 16 \sup_{Y, Z, Z'} \mathbb{E} | \mathbf{1}_{\Delta_{i,j} h(Y) \neq 0} \Delta_i h(Z)^2 \Delta_j h(Z')^2|,
\end{align}

\noindent where we have exchanged $\tilde{R}_i$ and $R_i'$ and used the Cauchy-Schwarz inequality as in~\eqref{eq:ijk}.

\noindent Similarly if $i = j \neq k$, the bound is 
\begin{align}
\notag  4 & \sup_{Y, Y', Z, Z'} \mathbb{E} | \tilde{\Delta}_i (\Delta_i (h(Y)) \Delta_i h(Y') \tilde{\Delta}_i (\Delta_k (h(Z)) \Delta_k h(Z')| \\
\notag & =  \quad 4 \sup_{Y, Y', Z, Z'}  \mathbb{E} | \Delta_i h(Y) \Delta_i h(Y') \Delta_i (\Delta_k (h(Z)) \Delta_k h(Z') |\\
\notag & =  \quad 4 \sup_{Y, Z, Z'} \mathbb{E} | \Delta_i h(Y)^2  \Delta_i (\Delta_k (h(Z)) \Delta_k h(Z') |\\
\label{eq:k} & \leq  \quad 8 \sup_{Y, Z, Z'} \mathbb{E} | \mathbf{1}_{\Delta_{i,k} h(Y) \neq 0} \Delta_i h(Z)^2 \Delta_k h(Z')^2|,
\end{align}

\noindent Finally, if $i = k \neq j$, the bound is by symmetry 
\begin{align}
\notag  4 & \sup_{Y, Y', Z, Z'} \mathbb{E} | \tilde{\Delta}_i (\Delta_j (h(Y)) \Delta_j h(Y') \tilde{\Delta}_i (\Delta_i (h(Z)) \Delta_i h(Z')| \\
\label{eq:j} &  \leq   \quad 8 \sup_{Y, Z, Z'} \mathbb{E} | \mathbf{1}_{\Delta_{i,j} h(Y) \neq 0} \Delta_i h(Z)^2 \Delta_j h(Z')^2|,
\end{align}

\noindent Combining~\eqref{eq:ijk},~\eqref{eq:i},~\eqref{eq:k} and~\eqref{eq:j} in~\eqref{eq:57} we finally get
\begin{align}
\notag  & \sum_{i = 0}^{|R| -1}   \mathbb{E} [ \tilde{\Delta}_i U_A]^2 \\
\notag &  \leq     16  \sum_{i = 0}^{|R| -1} \sum_{j, k \notin A}\bigg( \mathbf{1}_{i = j = k}  \mathbb{E} |\Delta_i h(R)|^4 + \mathbf{1}_{i \neq j \neq k} B_{|R|}(h)  \\
\notag & \quad \quad + (\mathbf{1}_{i \neq j = k} +\mathbf{1}_{i = k \neq j} ) B_{|R|}^{(k)}(h)  +  (\mathbf{1}_{i \neq j = k} + \mathbf{1}_{i = j \neq k} ) B_{|R|}^{(j)}(h) \bigg),
\end{align} 

\noindent where
\begin{align}
\notag B_{|R|}(h) & \coloneqq \sup_{Y, Y', Z, Z'} \mathbb{E} [\mathbf{1}_{\Delta_{i,j} h(Y) \neq 0, \Delta_{j,k} h(Y') \neq 0} |\Delta_j h(Z)|^2 |\Delta_k h(Z')|^2],\\ 
\notag B_{|R|}^{(k)}(h) & \coloneqq  \sup_{Y, Z, Z'} \mathbb{E} | \mathbf{1}_{\Delta_{i,k} h(Y) \neq 0} \Delta_i h(Z)^2 \Delta_k h(Z')^2|, \\
\notag B_{|R|}^{(j)} (h) &  \coloneqq  \sup_{Y, Z, Z'} \mathbb{E} | \mathbf{1}_{\Delta_{i,k} h(Y) \neq 0} \Delta_i h(Z)^2 \Delta_k h(Z')^2|.
\end{align}

\noindent Then~\eqref{eq:var_UA}, leads to a bound on the conditional variance $Var(\mathbb{E}[U|R])$, for $U = T_{|R|}(h)$ or $U = T_{|R|}'(h)$. 
\begin{prop}\label{prop:var_b} With the notation as above and for $U = T_{|R|}(h)$ or $U = T_{|R|}'(h)$:
\begin{align}
\notag & \sqrt{Var(\mathbb{E}[U|R])} \leq \frac{1}{\sqrt{2}} \sum_{\emptyset \subseteq A \subsetneq[|R|]} k_{|R|, A} \bigg(  \sum_{i = 0}^{|R| -1} \sum_{j, k \notin A}\bigg( \mathbf{1}_{i = j = k}  \mathbb{E} |\Delta_i h(R)|^4 + \mathbf{1}_{i \neq j \neq k} B_{|R|}(h) \\
\notag & \quad \quad \quad \quad  \quad \quad   \quad \quad +  (\mathbf{1}_{i \neq j = k} + \mathbf{1}_{i = k \neq j} ) B_{|R|}^{(k)}(h) +  (\mathbf{1}_{i \neq j = k} + \mathbf{1}_{i = j \neq k} ) B_{|R|}^{(j)}(h) \bigg) \bigg)^{1/2},
\end{align} 
\end{prop}

\noindent Note again that function $h$ is not symmetric and therefore the expression above cannot be simplified further in contrast to the case in~\cite{LRP}.

\section{Applications}\label{s:applications}

Although our framework was initially motivated by~\cite{HK} and finding a normal approximation result for the length of the longest common subsequences in dependent random words, some applications to stochastic geometry are presented below. Our methodology can be applied to other related settings, in particular the variant of the occupancy problem introduced in the recent article~\cite{gkp-2020} (see Remark~\ref{rem:occupancy}).

\subsection{Covering process}

Let $(K, \mathcal{K})$ be the space of compact subsets of $\mathbb{R}^d$, endowed with the hit-and-miss topology. Let $E_n$ be a cube of volume $n$, and $C_1, \ldots, C_n$ be random variables in $E_n$ called \emph{germs}. In the iid setting of~\cite{LRP} each $C_i$ is sampled uniformly and independently in $E_n$, i.e., if $T \subset E_n$ with measure $|T|$,
\begin{align}
\notag \mathbb{P}(C_i \in T) = \frac{|T|}{n},
\end{align} 
\noindent for all $i \in \{1, \ldots, n\}$.

\noindent Here, we consider $C_1, \ldots, C_n$, generated by a hidden Markov model in the following way. Let $Z_1, \ldots, Z_n$ be an aperiodic irreducible Markov chain on a finite state space $\mathcal{S}$. Each $s \in \mathcal{S}$ is associated with a measure $m_s$ on $E_n$. Then for each measurable $T \subseteq E_n$, 
\begin{align}
\notag \mathbb{P} (C_i \in T | Z_i = s) = m_s(T).
\end{align}  
\noindent Assume that there are constants $0 < c_m \leq c_M$ such that for any $s \in \mathcal{S}$ and measurable $T \subseteq E_n$, 
\begin{align}
\notag \frac{c_m |T|}{n} \leq m_s(T) \leq \frac{c_M |T|}{n}.
\end{align}
\noindent Note that $c_m = c_M = 1$ recovers the setting of~\cite{LRP}.

\noindent  Let $K_1, \ldots, K_n$ be  compact sets (\emph{grains}) with $Vol (K_i) \in (V_1, V_2)$ (absolute constants) for $i = 1,\ldots, n$.  Let $X_i = C_i + K_i$ for $i = 1, \ldots, n$ be the \emph{germ-grain} process. Consider the closed set formed by the union of the grains translated by the grain
\begin{align}
\notag F_n = \left( \bigcup_{k =1}^n X_K\right) \cap E_n.
\end{align}  

\noindent We are interested in the volume covered by $F_n$, 
\begin{align}
\notag f_V(X_1, \ldots, X_n) = Vol(F_n),
\end{align}

\noindent and the number of isolated grains
\begin{align}
\notag f_I(X_1, \ldots, X_n) = \# \{ k : X_k \cap X_j \cap E_n = \emptyset, k \neq j \}.
\end{align}

\begin{thm}\label{thm:cov} Let $\mathcal{N}$ be a standard normal random variable. Then, for all $n \in \mathbb{N}$,
\begin{align}
\label{eq:thm_cov_vol} d_K\left(\frac{f_V - \mathbb{E} f_V}{\sqrt{Var f_V}} , \mathcal{N} \right) \leq C \left( \frac{ n(\ln n)^3}  {\sqrt{Var(f_V)^3}} +   \frac{ n^{1/2} (\ln n)^4}{ Var(f_V)} \right),  \\ 
\label{eq:thm_cov_iso} d_K\left(\frac{f_I - \mathbb{E} f_I}{\sqrt{Var f_I}} , \mathcal{N} \right) \leq C \left( \frac{ n(\ln n)^3}  {\sqrt{Var(f_I)^3}} +   \frac{n^{1/2} (\ln n)^4} {Var(f_I)}\right),
\end{align}
\noindent for some constant $C > 0$, independent of $n$.
\end{thm}

 The study of the order of growth of $Var f_I$ and $Var f_V$ is not really the scope of the current paper.  In the independent case, there are constants $0 < c_V \leq C_V$, such that $c_V n \leq Var f_V \leq C_V n$, and  $c_V n \leq Var f_I \leq C_V n $, for $n$ sufficiently large (see~\cite[Theorem 4.4]{KM}). In our dependent setting a variance lower bound of order $n$ will thus provide a rate of order $(\log n)^4/\sqrt{n}$.

\begin{proof}
\noindent Write $f_V(X_1, \ldots, X_n) = h(R_0, \ldots, R_{|\mathcal{S}|(n-1)})$ for a set of instructions $R$ defined as in Section~\ref{s:construction}. The volume of each grain is bounded by $V_2$, so  $f_V$ is Lipschitz with constant $V_2$.  Proposition~\ref{thm:wass_hmm} holds, and from Proposition~\ref{prop:der_tail}, the non-variance terms in the bounds in Proposition~\ref{thm:wass_hmm} are bounded by  $C (\ln n)^3 / \sqrt{n}$. Here and below, $C$ is a constant, independent of $n$, which can vary from line to line. Indeed, for instance,
\begin{align}
\notag \frac{1}{4 \sigma^3} \sum_{j = 0}^{|R| - 1} \sqrt{ \mathbb{E} |\Delta_j h(R)|^6} \leq & \quad  C  Var(f_V)^{-3/2} (|S|(n-1) + 1) (\ln n )^3  \\
\label{eq:cov_mom}  \leq & \quad C n(\ln n)^3  / Var(f_V)^{3/2}.
\end{align}

\noindent To analyze the bound on the variance terms given by Proposition~\ref{prop:var_b} first note that
\begin{align}
\notag \sum_{i = 0}^{|R| -1} \sum_{j, k \notin A} \mathbf{1}_{i = j = k} \mathbb{E} | \Delta_i h(R)|^4 \leq Cn (\ln n)^4,
\end{align}

\noindent using Proposition~\ref{prop:der_tail}.

\noindent Next, we analyze
\begin{align}
\label{eq:brh_fv} B_{|R|}(h) & \coloneqq \sup_{Y, Y', Z, Z'} \mathbb{E} [\mathbf{1}_{\Delta_{i,j} h(Y) \neq 0, \Delta_{j,k} h(Y') \neq 0} |\Delta_j h(Z)|^2 |\Delta_k h(Z')|^2].
\end{align}

\noindent Let $E$ be the event that at least one of the  perturbations of the instructions in~\eqref{eq:brh_fv}  yields a difference in more than $K$ points. By Proposition~\ref{prop:der_tail}, there is $\ep > 0$, such that $\mathbb{P}( E) \leq (1 - \ep)^K$. Then, by the Lipschitz properties of $h$, 
\begin{align}
\notag    \mathbb{E} & [\mathbf{1}_{\Delta_{i,j} h(Y) \neq 0, \Delta_{j,k} h(Y') \neq 0} |\Delta_j h(Z)|^2 |\Delta_k h(Z')|^2] \\
\notag  & =  \quad  \mathbb{E} [\mathbf{1}_{\Delta_{i,j} h(Y) \neq 0, \Delta_{j,k} h(Y') \neq 0} |\Delta_j h(Z)|^2 |\Delta_k h(Z')|^2 \mathbf{1}_{E}]  \\ 
\notag & \quad  \quad  + \mathbb{E} [\mathbf{1}_{\Delta_{i,j} h(Y) \neq 0, \Delta_{j,k} h(Y') \neq 0} |\Delta_j h(Z)|^2 |\Delta_k h(Z')|^2 \mathbf{1}_{E^c}] \\
\notag &  \leq  \quad  \mathbb{E} [\mathbf{1}_{\Delta_{i,j} h(Y) \neq 0, \Delta_{j,k} h(Y') \neq 0} |\Delta_j h(Z)|^2 |\Delta_k h(Z')|^2 \mathbf{1}_{E^c}]  + Cn^4 (1- \ep)^K \\
\label{eq:brh_fv_1} & \leq    \quad  CK^4 \mathbb{E} [ \mathbf{1}_{\Delta_{i,j} h(Y) \neq 0, \Delta_{j,k} h(Y') \neq 0}  \mathbf{1}_{E^c}]  + Cn^4 (1- \ep)^K.
\end{align}

\noindent If $S(Y)$ is the set of points generated by the instructions $Y$ and $S(Y^i)$ - the set of points generated by $Y$ after the perturbation of $Y_i$, let 
\begin{align}
\notag S_1 \coloneqq & S(Y) \Delta S(Y^i),
\end{align}
\noindent where $\Delta$ is the symmetric difference operator. Similarly, let
\begin{align}
\notag S_2 \coloneqq & S(Y) \Delta S(Y^j), \\
\notag S_3 \coloneqq & S(Y') \Delta S((Y')^i), \\
\notag S_4 \coloneqq & S(Y') \Delta S((Y')^j).
\end{align}
\noindent Note that, conditioned on $E^c$, $|S_i| \leq 2K$ for $i = 1, 2,3, 4$. Furthermore, if $s_1 \cap s_2 = \emptyset$, for all $(s_1, s_2) \in (S_1, S_2)$, then $\Delta_{i, j} h(Y) = 0$. Then
\begin{align}
\notag \mathbf{1}_{\Delta_{i,j} h(Y)} \leq \sum_{(s_1, s_2) \in (S_1, S_2) } \mathbf{1}_{s_1 \cap s_2 \neq \emptyset}.
\end{align}
\noindent This bound is meaningful if the sets $S_1$ and $S_2$ are disjoint sets of random variables. Conditioned on $E^c$, this is the case if $|i - j| \geq |R|K$. We introduce events $E_1, E_2$ and $E_3$ corresponding to $0, 1$, or $2$ of the conditions $\{ |i - j| \leq |R| K, |j- k| \leq |R| K\}$ holding, respectively. The events $E_1, E_2,$ and $E_3$ are deterministic. Then, 
\begin{align}
\notag  \mathbb{E}&  [ \mathbf{1}_{\Delta_{i,j} h(Y) \neq 0, \Delta_{j,k} h(Y') \neq 0}  \mathbf{1}_{E^c}] = \mathbb{E} [ \mathbf{1}_{\Delta_{i,j} h(Y) \neq 0, \Delta_{j,k} h(Y') \neq 0}  \mathbf{1}_{E^c} ( \mathbf{1}_{E_1} + \mathbf{1}_{E_2} + \mathbf{1}_{E_3})] 
\end{align} 

\noindent First, we use the trivial bound $\mathbf{1}_{\Delta_{i,j} h(Y) \neq 0, \Delta_{j,k} h(Y') \neq 0}  \leq 1$, to get
\begin{align}
\label{eq:brh_fv_e1}  \mathbb{E}&  [ \mathbf{1}_{\Delta_{i,j} h(Y) \neq 0, \Delta_{j,k} h(Y') \neq 0}  \mathbf{1}_{E^c} \mathbf{1}_{E_1} ] \leq \mathbf{1}_{E_1}.
\end{align} 

\noindent Then, for the term with $\mathbf{1}_{E_3}$, 
\begin{align}
\notag  \mathbb{E}&  [ \mathbf{1}_{\Delta_{i,j} h(Y) \neq 0, \Delta_{j,k} h(Y') \neq 0}  \mathbf{1}_{E^c} \mathbf{1}_{E_3} ]   \leq  \mathbf{1}_{E_3} \mathbb{E} \left[ \sum_{(s_1, s_2) \in (S_1, S_2) } \sum_{(s_3, s_4) \in (S_3, S_4) } \mathbf{1}_{s_1 \cap s_2 \neq \emptyset, s_3 \cap s_4 \neq \emptyset}\right].
\end{align} 

\noindent To bound $\mathbb{E} [ \mathbf{1}_{s_1 \cap s_2 \neq \emptyset, s_3 \cap s_4 \neq \emptyset}]$, condition on $s_2, s_3$ and the values of all hidden variables $H$. Then, since $S_1$ and $S_4$ are disjoint we have independence, 
\begin{align}
\notag \mathbb{E}  [ \mathbf{1}_{s_1 \cap s_2 \neq \emptyset, s_3 \cap s_4 \neq \emptyset}]  = &  \mathbb{E} [ \mathbb{E} [ \mathbf{1}_{s_1 \cap s_2 \neq \emptyset, s_3 \cap s_4 \neq \emptyset} | s_2, s_3, H]] \\
\notag = &  \mathbb{E} [ \mathbb{E} [ \mathbf{1}_{s_1 \cap s_2 \neq \emptyset} |s_2, s_3, H] \mathbb{E} [ \mathbf{1}_{s_1 \cap s_2 \neq \emptyset} | s_2, s_3, H] ] \\
\notag \leq & \left(\frac{c_M V_2}{n}\right)^2.
\end{align}

\noindent Therefore, 
\begin{align}
\label{eq:brh_fv_e3}  \mathbb{E}&  [ \mathbf{1}_{\Delta_{i,j} h(Y) \neq 0, \Delta_{j,k} h(Y') \neq 0}  \mathbf{1}_{E^c} \mathbf{1}_{E_3} ] \leq \mathbf{1}_{E_3} C K^4/n^2,
\end{align} 
\noindent for some $C > 0$, independent of $K$ and $n$, and where we have used that $|S_i| \leq 2K$ for $i = 1,2,3,4$.

\noindent Finally, for the term with $E_2$, we may assume that $|i - j| \geq |R|K$, since the case $|j - k| \geq |R|K$ is identical. Write, using the trivial bound on $\mathbf{1}_{\Delta_{j,k} h(Y')} \neq 0$, 
\begin{align}
\notag  \mathbb{E}&  [ \mathbf{1}_{\Delta_{i,j} h(Y) \neq 0, \Delta_{j,k} h(Y') \neq 0}  \mathbf{1}_{E^c} \mathbf{1}_{E_2} ]   \leq  \mathbf{1}_{E_3} \mathbb{E} \left[ \sum_{(s_1, s_2) \in (S_1, S_2) }  \mathbf{1}_{s_1 \cap s_2 \neq \emptyset}\right].
\end{align} 
\noindent Next, as before, 
\begin{align}
\notag \mathbb{E} [  \mathbf{1}_{s_1 \cap s_2 \neq \emptyset}] = \mathbb{E}[ \mathbb{E} [  \mathbf{1}_{s_1 \cap s_2 \neq \emptyset} | s_2, H]] \leq \frac{c_M V_2}{n}. 
\end{align} 
\noindent Then, 
\begin{align}
\label{eq:brh_fv_e2}  \mathbb{E}&  [ \mathbf{1}_{\Delta_{i,j} h(Y) \neq 0, \Delta_{j,k} h(Y') \neq 0}  \mathbf{1}_{E^c} \mathbf{1}_{E_2} ] \leq \mathbf{1}_{E_2} C K^2/n,
\end{align} 
\noindent Then, combining~\eqref{eq:brh_fv_1},~\eqref{eq:brh_fv_e1}, ~\eqref{eq:brh_fv_e2} and ~\eqref{eq:brh_fv_e3}, we get the following bound of~\eqref{eq:brh_fv}, 
\begin{align}
\notag B_{|R|}(h) \leq C( \mathbf{1}_{E_1} K^4 + \mathbf{1}_{E_2} K^6/n + \mathbf{1}_{E_3} K^8/n^2 + n^4 (1 - \ep)^K).
\end{align}   

\noindent Then, 
\begin{align}
\notag  \sum_{i = 0}^{|R| -1} & \sum_{j, k \notin A}  \mathbf{1}_{i \neq j \neq k} B_{|R|}(h) \\
\notag \leq & \quad  C( n K^5 + n^2 K^7 /n + n^3 K^8 /n^2 + n^7 (1 - \ep)^K ) \\
\notag \leq & \quad Cn (\ln n)^8, 
\end{align}
\noindent when we choose $K = c \ln n$ for a suitable $ c> 0$, independent of $n$.

\noindent Similarly, 
\begin{align}
\notag B_{|R|}^{(k)}(h) \leq C (\ln n)^4/n,\\
\notag B_{|R|}^{(j)}(h) \leq C (\ln n)^4/n.
\end{align}

\noindent and
\begin{align}
\notag  \sum_{i = 0}^{|R| -1} \sum_{j, k \notin A} (\mathbf{1}_{i \neq j = k} + \mathbf{1}_{i = k \neq j} ) B_{|R|}^{(k)}(h)  \leq Cn^2 (\ln n)^4/n = Cn (\ln n)^4, \\
\notag  \sum_{i = 0}^{|R| -1} \sum_{j, k \notin A}  (\mathbf{1}_{i \neq j = k} + \mathbf{1}_{i = j \neq k} ) B_{|R|}^{(j)}(h) \leq Cn^2 (\ln n)^4/n = Cn (\ln n)^4.
\end{align}

\noindent The bounds on the variance terms in Proposition~\ref{prop:var_b} become
\begin{align}
\notag  \sqrt{Var(\mathbb{E}[U|R])} \leq & \frac{1}{\sqrt{2}} \sum_{A \subsetneq[|R|]} k_{|R|, A} \bigg(  C n (\ln n)^4 + Cn (\ln n)^8 + 2 C n (\ln n)^4 \bigg)^{1/2}\\
\label{eq:cov_var_fv} \leq  & C \sqrt{n} (\ln n)^4.
\end{align}

\noindent Then,~\eqref{eq:thm_cov_vol} follows from~\eqref{eq:cov_var_fv},~\eqref{eq:cov_mom} and Theorem~\ref{thm:wass_hmm}.

\noindent The proof of~\eqref{eq:thm_cov_iso} is more involved since the function $f_I$ is not Lipschitz. Write, abusing notation, $f_I(X_1, \ldots, X_n) = h(R_0, \ldots, R_{|S|(n-1)})$ for a set of instructions $R$ as in Section~\ref{s:construction}. Proposition~\ref{thm:wass_hmm} holds and, like our analysis for $f_V$, we proceed by estimating the non-variance terms in the bounds. We first prove that, for any $t = 1, 2, \ldots$ and $i \in \{0, \ldots, |S|(n-1)\}$, 
\begin{align}
\label{eq:per_fi} \mathbb{E} |\Delta_i h |^t  \leq C (\ln n)^t, 
\end{align}
\noindent where $C = C(t) > 0$.

\noindent As in the proof of Proposition~\ref{prop:der_tail}, the sequence of instructions $R^i$ may give rise to a different realization $(Z', X')$. Indeed, if instruction $R_i$ determines $(Z_j, X_j)$ and $R_i'$ determines $(Z_j', X_j')$, it is possible that $(Z_j, X_j) \neq (Z_j', X_j')$. Let $s \geq 0$ be the smallest integer (possibly $s = \infty$) such that $Z_{j + s} = Z_{j+s}'$. Then, as in~\eqref{eq:s_decay}, there is $\ep > 0$, such that for $K \in \mathbb{N}$, 
\begin{align}
\notag \mathbb{P}(s \geq K) \leq (1 - \ep)^K.
\end{align}
\noindent Fix $K$, and let $E$ be the event, corresponding to $\{s \geq K\}$. Using the trivial bound $|h(R)| \leq n$, and thus $|\Delta_i h(R)| \leq 2n$,
\begin{align}
\notag  \mathbb{E} |\Delta_i h |^t  = & \quad  \mathbb{E} [ |\Delta_i h |^t \mathbf{1}_E] +  \mathbb{E} [|\Delta_i h |^t \mathbf{1}_{E^c} ] \\
\label{eq:per_fi_1} \leq & \quad (2n)^t (1 - \ep)^K + \mathbb{E} [|\Delta_i h |^t \mathbf{1}_{E^c} ]. 
\end{align}
\noindent Let $S(R)$ be the set of points generated by the sequence of instructions $R$, and $S(R^j)$ - be the points generated by $R$ after the perturbation of $R_j$. Set $S = S(R) \Delta S(R^j)$ for the symmetric difference and $S^c = S(R) \cap S(R^j)$. Note that $E^c$ implies that $|S| \leq 2K$. Furthermore, 
\begin{align}
\notag |\Delta_i h| \leq \sum_{s \in S} \sum_{x \in S^c} \mathbf{1}_{s \cap x \neq \emptyset},
\end{align}
\noindent and 
\begin{align}
\notag |\Delta_i h|^t \leq \sum_{(s_1, \ldots, s_t) \in S^t} \sum_{(x_1, \ldots, x_t) \in (S^c)^t}  \prod_{j, \ell = 1}^t \mathbf{1}_{s_j \cap x_{\ell} \neq \emptyset},
\end{align}
\noindent To estimate~\eqref{eq:per_fi_1}, we need to evaluate $\mathbb{E}[ \prod_{j, \ell = 1}^t \mathbf{1}_{s_j \cap x_{\ell} \neq \emptyset}]$, and to do so we proceed as in~\cite{LRP}  by studying the shape of the relations of $(s_j, x_{\ell})_{j, \ell \in \{1, \ldots, t\}}$. 

\noindent Identify the set $(s_j, x_{\ell})_{j, \ell \in \{1, \ldots, t\}}$ with the edges of the graph $G$, whose vertices correspond to $(s_j)_{j \in \{1, \ldots, t\}}$ and $(x_{\ell})_{ \ell \in \{1, \ldots, t\}}$. In particular, if $s_{j_1} = s_{j_2}$, for some $j_1 \neq j_2$, we identify them with the same point in the graph $G$. Conditioned on the realization of the hidden chain $Z$, we have independence. Then,  if $G$ is a tree, fix a root and condition recursively on vertices at different distances from the root. By  the restrictions on the volume of the grain and the sampling distribution, 
\begin{align}
\notag  \mathbb{E}\left[ \prod_{j, \ell = 1}^t \mathbf{1}_{s_j \cap x_{\ell} \neq \emptyset} \bigg| Z = z^n\right] \leq \left(\frac{c_M V_2}{n}\right)^{|E(G)|},
\end{align}
\noindent where $|E(G)|$ is the number of edges in the graph $G$. Furthermore, 
\begin{align}
\notag  \mathbb{E}\left[ \prod_{j, \ell = 1}^t \mathbf{1}_{s_j \cap x_{\ell} \neq \emptyset} \right] \leq \left(\frac{c_M V_2}{n}\right)^{|E(G)|}.
\end{align}
\noindent Note that the same result holds if $G$ is a graph without cycles, i.e., a collection of disjoint trees.
\noindent In general, $G$ might have cycles. Let $T$ be a subgraph of $G$ that contains no cycles. Then, 
\begin{align}
\notag   \prod_{j, \ell = 1}^t \mathbf{1}_{s_j \cap x_{\ell} \neq \emptyset} \leq \prod_{e = (e_1,e_2) \in E(T)} \mathbf{1}_{e_1 \cap e_2 \neq \emptyset},
\end{align}
\noindent where the product on the right-hand side runs over the edges $e = (e_1, e_2)$ of the graph $T$, with $e_1 \in S$ and $e_2 \in S^c$. 
\noindent Let $|s|$ be the number of distinct vertices in $(s_1, \ldots, s_t)$, and similarly let $|x|$ be the number for $(x_1, \ldots, x_t)$. The graph $G$ is complete bipartite with $|s| + |x|$ vertices. We can find a subgraph $T$ of $G$, also with $|s| + |x|$ vertices and no cycles. Then, 
\begin{align}
\notag \mathbb{E} [ |\Delta_i h|^t \mathbf{1}_{E^c} ] \leq & \quad \mathbb{E} \left[  \mathbf{1}_E^c \sum_{(s_1, \ldots, s_t) \in S^t} \sum_{(x_1, \ldots, x_t) \in (S^c)^t}  \prod_{j, \ell = 1}^t \mathbf{1}_{s_j \cap x_{\ell} \neq \emptyset}\right] \\
\notag  = & \quad \mathbb{E} \left[  \mathbf{1}_E^c \sum_{a, b = 1}^t \sum_{ \substack{(s_1, \ldots, s_t) \in S^t,\\ |s| = a}} \sum_{\substack{(x_1, \ldots, x_t) \in (S^c)^t,\\ |x| = b}}  \prod_{j, \ell = 1}^t \mathbf{1}_{s_j \cap x_{\ell} \neq \emptyset}\right] \\
\notag \leq & \quad \mathbb{E} \left[ \mathbf{1}_E^c \sum_{a, b = 1}^t  C_t |S|^a |S^c|^b \left(\frac{c_M V_2}{n}\right)^{a + b -1} \right] \\
\notag \leq & \quad C_t K^r,
\end{align}
\noindent where $C_t > 0$ is a constant depending on $t$, and where we have used that $|S| \leq 2K$ and $|S^c| \leq 2n$.

\noindent Letting $K = c\ln n$, for a suitable $c > 0$,~\eqref{eq:per_fi_1} implies~\eqref{eq:per_fi} as desired. Therefore, for the non-variance term in Proposition~\ref{thm:wass_hmm}, we have
\begin{align}
\label{eq:noncov_fi} \frac{1}{4 \sigma^3} \sum_{j = 0}^{|R| - 1} \sqrt{\mathbb{E} | \Delta_j h(R)|^6} + \frac{\sqrt{2 \pi} } {16 \sigma^3} \sum_{j = 0}^{|R| - 1} \mathbb{E} |\Delta_j h(R)|^3\leq C n \left( \frac{\ln n }{  \sqrt{ Var(f_I) } }\right)^3.
\end{align}  

\noindent We are left to analyze the bound on the variance terms given by Proposition~\ref{prop:var_b}. First, note that
\begin{align}
\notag \sum_{i = 0}^{|R| -1} \sum_{j, k \notin A} \mathbf{1}_{i = j = k} \mathbb{E} | \Delta_i h(R)|^4 \leq Cn (\ln n)^4,
\end{align}

\noindent Next, we analyze
\begin{align}
\label{eq:brh_fi} B_{|R|}(h) \coloneqq \sup_{Y, Y', Z, Z'} \mathbb{E} [ \mathbf{1}_{\Delta_{i,j} h(Y) \neq 0,\Delta_{j,k} h(Y') \neq 0} |\Delta_j h(Z)|^2 |\Delta_k h(Z')|^2 ], 
\end{align}
\noindent where the supremum is taken over recombinations $Y, Y', Z, Z'$ of $R, R', \tilde{R}$.
\noindent As before, let $E$ be the event that all perturbations of instructions in~\eqref{eq:brh_fi} propagate at most $K$ levels. We have that $\mathbb{P}(E^c) \leq (1 - \ep)^K$, for some $\ep \in (0,1)$. Using the trivial bound $|h(Y)| \leq n$, 
\begin{align}
\notag  \mathbb{E} & [ \mathbf{1}_{\Delta_{i,j} h(Y) \neq 0,\Delta_{j,k} h(Y') \neq 0} |\Delta_j h(Z)|^2 |\Delta_k h(Z')|^2 ] \\
\notag = & \quad  \mathbb{E} [ \mathbf{1}_{\Delta_{i,j} h(Y) \neq 0,\Delta_{j,k} h(Y') \neq 0} |\Delta_j h(Z)|^2 |\Delta_k h(Z')|^2 \mathbf{1}_E] \\
\notag & +  \mathbb{E} [ \mathbf{1}_{\Delta_{i,j} h(Y) \neq 0,\Delta_{j,k} h(Y') \neq 0} |\Delta_j h(Z)|^2 |\Delta_k h(Z')|^2 \mathbf{1}_{E^c}]\\
\label{eq:brh_fi_1} \leq & \quad  \mathbb{E} [ \mathbf{1}_{\Delta_{i,j} h(Y) \neq 0,\Delta_{j,k} h(Y') \neq 0} |\Delta_j h(Z)|^2 |\Delta_k h(Z')|^2 \mathbf{1}_E] + 4n^4 (1 -\ep)^K.
\end{align}

\noindent Let $S(Y^i)$ be the set of points generated by the sequence of instructions $Y$ after the perturbation of $Y_i$. Let $S$ be the set of all points in the expectation above, and furthermore let
\begin{align}
\notag  S_1 \coloneqq & S(Y) \Delta S(Y^i), \quad S_2 \coloneqq  S(Y)  \Delta S(Y^j),  \\
\notag S_3 \coloneqq & S(Y') \Delta S((Y')^j), \quad S_4 \coloneqq  S(Y') \Delta S((Y')^k), \\
\notag S_5 \coloneqq & S(Z) \Delta S(Z^j ), \quad  S_6 \coloneqq  S(Z') \Delta S(Z^k),  
\end{align}
\noindent where $\Delta$ is the symmetric difference operator. Conditioned on $E$, $|S_i| \leq 2K$, for $i = 1, \ldots, 6$ and $ |S| \leq 10n$.  

\noindent Conditioned on $E$, if $j - i \leq |R| K$, the perturbation in $i$ might be propagating past the position, corresponding to instruction $j$, leading to difficulties in the analysis of $\Delta_{i,j} h(Y)$. This is why, we condition further on the events $E_1, E_2, E_3$ corresponding to respectively $0, 1,$ or $2$ of the conditions $\{|i - j| \geq |R|K, |j - k| \geq |R|K\}$ holding true. Note that $E_1, E_2$ and $E_3$ are deterministic. 

\noindent If $E_1$ holds, use the trivial bound $ \mathbf{1}_{\Delta_{i,j} h(Y) \neq 0,\Delta_{j,k} h(Y') \neq 0} \leq 1$, leading to
\begin{align}
\notag \mathbb{E} & [ \mathbf{1}_{\Delta_{i,j} h(Y) \neq 0,\Delta_{j,k} h(Y') \neq 0} |\Delta_j h(Z)|^2 |\Delta_k h(Z')|^2 \mathbf{1}_E \mathbf{1}_{E_1}]  \\
\notag \leq & \quad \mathbb{E} [  |\Delta_j h(Z)|^2 |\Delta_k h(Z')|^2 \mathbf{1}_E \mathbf{1}_{E_1}] \\
\label{eq:brh_fi_e1} \leq & \quad \mathbf{1}_{E_1}  C K^4,
\end{align} 
\noindent using the Cauchy-Schwarz inequality.  

\noindent Conditioned on $E_3$, the sets $S_1, S_2 \cup S_3$ and $S_4$ are pairwise disjoint. Next, in similarity to an argument presented in~\cite{LRP}, if $s_1 \cap s = \emptyset$ and $s_2 \cap s = \emptyset$, for all $(s_1, s_2, s) \in (S_1, S_2, S)$, then $\Delta_{i,j} h(Y) = 0$. Therefore,  
\begin{align}
\notag  \mathbf{1}_{\Delta_{i,j} h(Y) \neq 0} \leq \sum_{ \substack{ s_1 \in S_1 \\ s_2 \in S_2 } } \sum_{s \in S} \mathbf{1}_{s_1 \cap s \neq \emptyset, s_2 \cap s \neq \emptyset},
\end{align}
\noindent and also
\begin{align}
\notag  \mathbf{1}_{\Delta_{j,k} h(Y') \neq 0} \leq \sum_{\substack{ s_3 \in S_3 \\ s_4 \in S_4} } \sum_{s \in S} \mathbf{1}_{s_3 \cap s \neq \emptyset, s_4 \cap s \neq \emptyset},
\end{align}
\noindent  Furthermore, 
\begin{align}
\notag  |\Delta_j h(Z) | \leq \sum_{s_5 \in S_5} \sum_{s \in S} \mathbf{1}_{s_5 \cap s \neq \emptyset},
\end{align}
\noindent and 
\begin{align}
\notag  |\Delta_k h(Z') | \leq \sum_{s_6 \in S} \sum_{s \in S} \mathbf{1}_{s_6 \cap s \neq \emptyset}.
\end{align}
\noindent Therefore, 
\begin{align}
\notag  \mathbb{E} & [ \mathbf{1}_{\Delta_{i,j} h(Y) \neq 0,\Delta_{j,k} h(Y') \neq 0} |\Delta_j h(Z)|^2 |\Delta_k h(Z')|^2 \mathbf{1}_E \mathbf{1}_{E_3} ]  \\
\notag \leq & \mathbb{E} \Bigg[ \Bigg( \sum_{ \substack {(s_1, s_2, s_3, s_4) \in (S_1, S_2, S_3, S_4) \\ (s', s'') \in S^2} }  \mathbf{1}_{\substack{ s_1 \cap s' \neq \emptyset, s_2 \cap s' \neq \emptyset,\\ s_3 \cap s'' \neq \emptyset, s_4 \cap s'' \neq \emptyset}}\Bigg) \\
\notag & \quad \quad \cdot \Bigg( \sum_{s_5 \in S_5} \sum_{s \in S} \mathbf{1}_{s_5 \cap s \neq \emptyset}\Bigg)^2 \Bigg( \sum_{s_6 \in S_6} \sum_{s \in S} \mathbf{1}_{s_6 \cap s \neq \emptyset}\Bigg)^2 \mathbf{1}_E \mathbf{1}_{E_3} \Bigg] \\
\label{eq:brh_long} \leq & \mathbb{E} \Bigg[ \sum_{ \substack {(s_1, \ldots, s_4) \in (S_1, \ldots, S_4) \\ (s_5, \ldots, s_8) \in S_{56}^4 } } \sum_{\substack{ (s', s'') \in S^2 \\ (s_5', \ldots, s_8') \in S^4}}  \mathbf{1}_{\substack{ s_1 \cap s' \neq \emptyset, s_2 \cap s' \neq \emptyset,\\ s_3 \cap s'' \neq \emptyset, s_4 \cap s'' \neq \emptyset}} \prod_{a, b = 5}^8 \mathbf{1}_{s_a \cap s_b' \neq \emptyset}  \mathbf{1}_E \mathbf{1}_{E_3} \Bigg],  
\end{align}
\noindent where $S_{56} = S_5 \cup S_6$ and $|S_{56}| \leq 4K$, conditioned on $E$.

\noindent To evaluate the summand expression we use the graph representation. Let $E_{\ell}$ be the event that  there are $\ell$ distinct points among $s',s'', s_5', \ldots, s_8'$, different from $s_1, \ldots, s_8$. Note that $\ell \in [0, 6]$. Conditioned on $E_{\ell}$, we can find a subgraph with no cycles and $\ell + 2$ edges, of the graph with edges $\{  \{s_1, s'\}, \{s_2, s'\}, \{s_3, s''\}, \{s_4, s''\} \} \cup \{ \{s_a, s_b'\}: a, b \in [5,8]\}$. Indeed, note that there are at least $3$ different points among $s_1, \ldots, s_4$. Next, if there are $x$ points present among $s', s''$ and $\ell - x$ points among $s_5', \ldots, s_8'$, we can find a subgraph with no cycles with at least $\ell - x$ edges among $\{ \{s_a, s_b'\}: a, b \in [5,8]\}$ and $x + 2$ edges among $\{  \{s_1, s'\}, \{s_2, s'\}, \{s_3, s''\}, \{s_4, s''\} \}$.

\noindent Then, if we further condition on the values of the hidden variables $H$, we get, by independence,
\begin{align}
\notag  \mathbb{E} \left[ \mathbf{1}_{\substack{ s_1 \cap s' \neq \emptyset, s_2 \cap s' \neq \emptyset,\\ s_3 \cap s'' \neq \emptyset, s_4 \cap s'' \neq \emptyset}} \prod_{a, b = 5}^8 \mathbf{1}_{s_a \cap s_b' \neq \emptyset}  \mathbf{1}_E \mathbf{1}_{E_3} \mathbf{1}_{E_{\ell}} \bigg| H\right]   \leq \mathbf{1}_{E_3}    \left( \frac{c_M V_2}{n}\right)^{\ell + 2}.
\end{align}

\noindent Then,~\eqref{eq:brh_long} is further bounded by
\begin{align}
\label{eq:brh_fi_e3}  \mathbf{1}_{E_3}  \sum_{\ell = 0}^6 (4K)^8 \binom{6}{\ell} (10n)^{\ell} \left( \frac{c_M V_2}{n}\right)^{\ell + 2} \leq \mathbf{1}_{E_3}  C  K^8 n^{-2},
\end{align}
\noindent for some $C > 0$, independent of $n$ and $K$.

\noindent Finally, assume that $E_2$ holds and that $|i - j| \geq |R| K$. The case $|j - k| \geq |R|K$ is identical. As above, using the trivial bound $\mathbf{1}_{\Delta_{j,k} h(Y') \neq 0} \leq 1$,
\begin{align}
\notag  \mathbb{E} & [ \mathbf{1}_{\Delta_{i,j} h(Y) \neq 0,\Delta_{j,k} h(Y') \neq 0} |\Delta_j h(Z)|^2 |\Delta_k h(Z')|^2 \mathbf{1}_E \mathbf{1}_{E_2} ]  \\
\notag \leq &  \mathbb{E}  [ \mathbf{1}_{\Delta_{i,j} h(Y) \neq 0} |\Delta_j h(Z)|^2 |\Delta_k h(Z')|^2 \mathbf{1}_E \mathbf{1}_{E_2} ]  \\
\notag \leq & \mathbb{E} \Bigg[ \Bigg( \sum_{ \substack {(s_1, s_2) \in (S_1, S_2) \\ s' \in S} }  \mathbf{1}_{ s_1 \cap s' \neq \emptyset, s_2 \cap s' \neq \emptyset}\Bigg) \Bigg( \sum_{s_5 \in S_5} \sum_{s \in S} \mathbf{1}_{s_5 \cap s \neq \emptyset}\Bigg)^2 \Bigg( \sum_{s_6 \in S_6} \sum_{s \in S} \mathbf{1}_{s_6 \cap s \neq \emptyset}\Bigg)^2 \mathbf{1}_E \mathbf{1}_{E_2} \Bigg] \\
\label{eq:brh_long_2} \leq & \mathbb{E} \Bigg[ \sum_{ \substack {(s_1 s_2) \in (S_1, S_2) \\ (s_5, \ldots, s_8) \in S_{56}^4 } } \sum_{\substack{ s' \in S \\ (s_5', \ldots, s_8') \in S^4}}  \mathbf{1}_{ s_1 \cap s' \neq \emptyset, s_2 \cap s' \neq \emptyset} \prod_{a, b = 5}^8 \mathbf{1}_{s_a \cap s_b' \neq \emptyset}  \mathbf{1}_E \mathbf{1}_{E_2} \Bigg].
\end{align}
\noindent Then, if we condition on $E_{\ell}$ and the values of the hidden variables $H$, we get
\begin{align}
\notag \mathbb{E} \Bigg[ \mathbf{1}_{ s_1 \cap s' \neq \emptyset, s_2 \cap s' \neq \emptyset} \prod_{a, b = 5}^8 \mathbf{1}_{s_a \cap s_b' \neq \emptyset}  \mathbf{1}_E \mathbf{1}_{E_2} \mathbf{1}_{E_{\ell}} | H \Bigg] \leq \mathbf{1}_{E_2} \left( \frac{c_M V_2}{n}\right)^{\ell + 1},
\end{align}
\noindent since in this case $s_1$ and $s_2$ are distinct and we can find a subgraph with $\ell + 1$ edges and no cycles. 

\noindent  Then,~\eqref{eq:brh_long_2} is  bounded by
\begin{align}
\label{eq:brh_fi_e2}  \mathbf{1}_{E_2}  \sum_{\ell = 0}^6 (4K)^6 \binom{6}{\ell} (10n)^{\ell} \left( \frac{c_M V_2}{n}\right)^{\ell + 1} \leq \mathbf{1}_{E_2}  C K^6  n^{-1},
\end{align}
\noindent for some $C > 0$.

\noindent We get the following bound on $B_{|R|} (h)$ using~\eqref{eq:brh_fi_1}, ~\eqref{eq:brh_fi_e1}, ~\eqref{eq:brh_fi_e2}, and~\eqref{eq:brh_fi_e3}, 
\begin{align}
\notag B_{|R|}(h) \leq C ( \mathbf{1}_{E_1} K^4 + \mathbf{1}_{E_2} K^6 /n + \mathbf{1}_{E_3} K^8 / n^2 + n^4 (1 - \ep)^K).
\end{align}

\noindent Then, 
\begin{align}
\notag  \sum_{i = 0}^{|R| - 1} & \sum_{j, k \notin A} \mathbf{1}_{i \neq j \neq k} B_{|R|} (h) \\ 
\notag \leq &  C ( n K^6 + n^2 K^7/n + n^3 K^8 /n^2 + n^7 (1 - \ep)^K) \\
\notag \leq & C n (\ln n)^8,
\end{align}
\noindent where we have chosen $K = c \ln n$ for a suitable $c > 0$, independent of $n$.
\noindent Finally, similar arguments yield, as in the case for $f_V$, 
\begin{align}
\notag B_{|R|}^{(k)}(h) \leq C (\ln n)^4/n,\\
\notag B_{|R|}^{(j)}(h) \leq C (\ln n)^4/n.
\end{align}

\noindent and
\begin{align}
\notag  \sum_{i = 0}^{|R| -1} \sum_{j, k \notin A} (\mathbf{1}_{i \neq j = k} + \mathbf{1}_{i = k \neq j} ) B_{|R|}^{(k)}(h)  \leq Cn^2 (\ln n)^4/n = Cn (\ln n)^4, \\
\notag  \sum_{i = 0}^{|R| -1} \sum_{j, k \notin A}  (\mathbf{1}_{i \neq j = k} + \mathbf{1}_{i = j \neq k} ) B_{|R|}^{(j)}(h) \leq Cn^2 (\ln n)^4/n = Cn (\ln n)^4.
\end{align}

\noindent The bounds on the variance terms in Proposition~\ref{prop:var_b} become
\begin{align}
\notag  \sqrt{Var(\mathbb{E}[U|R])} \leq & \frac{1}{\sqrt{2}} \sum_{A \subsetneq[|R|]} k_{|R|, A} \bigg(  C n (\ln n)^4 + Cn (\ln n)^8 + 2 C n (\ln n)^4 \bigg)^{1/2}\\
\label{eq:cov_var_fi} \leq  & C \sqrt{n} (\ln n)^4.
\end{align}

\noindent Then,~\eqref{eq:thm_cov_iso} follows from~\eqref{eq:cov_var_fi},~\eqref{eq:noncov_fi} and Proposition~\ref{thm:wass_hmm}.

\end{proof}

\subsection{Set approximation with random tessellations.} 

\noindent Let $K \subseteq [0,1]^d$ be compact, and $X$ be a finite collection of points in $K$. The Voronoi reconstruction, or the Voronoi approximation, of $K$ based on $X$ is given by
\begin{align}
\notag K^X \coloneqq \{ y \in \mathbb{R}^d: \text{ the closest point from } y \text{ in } X \text{ lies in } K\}.
\end{align} 

\noindent For $x \in [0,1]^d$, denote by $V(x; X)$ the Voronoi cell with nucleus $x$ among $X$, as
\begin{align}
\notag V(x; X) \coloneqq \{ y \in [0,1]^d : ||y - x|| \leq || y - x'||, \text{ for any } x' \in (X,x) \},
\end{align}
\noindent where $(X, x) = X \cup \{x\}$, and where, as usual, $||\cdot||$ is the Euclidean norm in $\mathbb{R}^d$. The volume approximation of interest is:
\begin{align}
\notag  \varphi(X) \coloneqq Vol(K^X) = \sum_i \mathbf{1}_{X_i \in K} Vol( V(X_i; X)).
\end{align}

\noindent In~\cite{LRP}, $X = (X_1, \ldots, X_n)$ is a vector of $n$ iid random variables uniformly distributed on $[0,1]^d$. Here, we consider $X_1, \ldots, X_n$, generated by a hidden Markov model in the following way. Let $Z_1, \ldots, Z_n$ be an aperiodic irreducible Markov chain on a finite state space $\mathcal{S}$. Each $s \in \mathcal{S}$ is associated with a measure $m_s$ on $[0,1]^d$. Then for each measurable $T \subseteq [0,1]^d$,
\begin{align}
\notag \mathbb{P}(X_i \in T | Z_i = s) = m_s(T).
\end{align} 
\noindent Assume, moreover, that there are constants $0 < c_m \leq c_M$, such that for any $s \in \mathcal{S}$ and measurable $T \subseteq [0,1]^n$, 
\begin{align}
\notag c_m\frac{ |T|}{n} \leq m_s(T) \leq c_M \frac{ |T|}{n}.
\end{align}

\noindent Recall the notions of Lebesgue-boundary of $K$  given by
\begin{flalign}
\notag  &\partial K  \coloneqq  \{ x \in [0,1]^d : Vol(B(x, \ep) \cap K) > 0 \text{ and } Vol(B(x, \ep) \cap K^c) > 0, \text{ for any } \ep > 0 \}, && 
\end{flalign}
\noindent and 
\begin{flalign}
\notag & \partial K^r \coloneqq  \{ x : d(x, \partial K) \leq r\},  \partial K_+^r \coloneqq  K^c \cap \partial K^r, && 
\end{flalign}
\noindent where $d(x, A)$ is the Euclidean distance from $x \in \mathbb{R}^d$ to $A \subseteq \mathbb{R}^d$.
 
\noindent Now, for $\beta > 0$, let 
\begin{align}
\notag \gamma (K,r, \beta) \coloneqq &  \int_{\partial K_+^r} \left( \frac{Vol(B(x, \beta r) \cap K)}{r^d}\right)^2 dx.
\end{align}

\noindent Next, recall that $K$ is said to satisfy the weak rolling ball condition if
\begin{align}
\label{eq:rbc} \gamma(K, \beta) \coloneqq \liminf_{r > 0} Vol(\partial K^r)^{-1} (\gamma(K, r, \beta) + \gamma(K^c, r, \beta) ) > 0.
\end{align}

\begin{thm}\label{thm:vor_main} Let $K \subseteq [0,1]^d$ be such that
\begin{align}
\notag Vol(\partial K^r) \leq S_+(K)r^{\alpha}, \quad r > 0, 
\end{align} 
\noindent for some $S_+(K), \alpha > 0$.  Then for $n,q \geq 1$,
\begin{align}
\label{eq:vor_mom} \mathbb{E} |\varphi(X) - \mathbb{E} \varphi(X) |^q \leq C_{d, q, \alpha}  S_+ (K) (\ln n)^{q} n^{-q/2 - \alpha /d},
\end{align}
\noindent for some $C_{d, q, \alpha} > 0$. If furthermore $K$ satisfies the weak rolling ball condition~\eqref{eq:rbc} and 
\begin{align}
\notag Vol(\partial K^r) \geq S_- (K) r^{\alpha}, \quad r > 0,
\end{align}
\noindent for some $S_- (K) > 0$, then for $n$ sufficiently large, 
\begin{align}
\label{eq:vor_var} C_d^- S_- (K) \gamma(K) \leq \frac{Var ( \varphi(K, X)) }{n^{-1 - \alpha/d} } \leq C_d^+ S_+ (K) C_{d, 2, \alpha}, 
\end{align}
\noindent for some $C_d^-, C_d^+ > 0$; and for every $\ep > 0$, there is $c_{\ep} > 0$ not depending on $n$ such that
\begin{align}
\label{eq:vor_clt} d_K \left( \frac{ \varphi(X) - \mathbb{E} \varphi(X) }{\sqrt{ Var (\varphi(X)) }}, \mathcal{N}\right) \leq c_{\ep} \frac{ (\log n)^{3 + \alpha/d + \ep}}{n^{1/2 - \alpha/2d}},
\end{align}
\noindent for $n \geq 1$.
\end{thm}

\noindent As in~\cite{LRP}, we split the proof of Theorem~\ref{thm:vor_main} into several results. To start, we show:

\begin{thm}\label{thm:vor_clt} Let $0 < \sigma^2 = Var( \varphi(X))$. Assume that $Vol(\partial K^r) \leq S_+(K) r^{\alpha}$ for some $S_+(K), \alpha > 0$. Then~\eqref{eq:vor_mom} holds, and for every $\ep > 0$ there is a constant $C$ not depending on $n$ such that, for $n \geq 1$, 
\begin{align}
\label{eq:thm_cov_vor} d_K \left( \frac{ \varphi(X) - \mathbb{E} \varphi(X) }{\sigma }, \mathcal{N}\right)\leq &  C\left( \frac{ (\ln n)^{5 + 2 \ep} }{ \sigma^2 n^{3/2 + \alpha/d}} +  \frac{ (\ln n)^3}{ \sigma^3 n^{3 + \alpha/d}}\right).
\end{align}

\end{thm}

\begin{proof}[Proof of Theorem~\ref{thm:vor_clt}] 

\noindent Recall that $x, y \in [0,1]^d$ are said to be Voronoi neighbors among the set $X$ if $V(x; X) \cap V(y; X) \neq \emptyset$. In general, the Voronoi distance $d_V(x, y;X)$ among $X$ of $x$ and $y$, is given by the smallest $k \geq 1$ such that there exist $x = x_0, x_1 \in X, \ldots, x_{k-1} \in X, x_k = y$ and $x_i, x_{i+1}$ are Voronoi neighbors for $i = 0, \ldots, k-1$.  

\noindent Denote by $v(x, y; X) = Vol\bigg(V(y; X) \cap V(x; (y, X)) \bigg)$, the volume that $V(y;X)$ loses when $x$ is added to $X$. Then, for $x \notin X$,
\begin{align}
\notag \varphi(X, x)  - \varphi(X) = \mathbf{1}_{x \in K} \sum_{y \in X \cap K^c} v(x, y; X) - \mathbf{1}_{x \in K^c} \sum_{y \in X \cap K} v(x, y; X). 
\end{align}

\noindent Let $R_k(x; X)$ be the distance from $x$ to the furthest point in the cell of a $k$th order Voronoi neighbor in $X$, i.e., 
for $X = (X_1, \ldots, X_n)$, 
\begin{align}
\notag R_k(x; X) = \sup \{ || y - x|| : y \in V(X_i; X), d_V(x, X_i; X) \leq k\},
\end{align}
\noindent with $R(x; X) \coloneqq R_1(x; X)$. If $x$ does not have $k$th order neighbors, take $R_k(x; X) = \sqrt{d}$. Then,
\begin{align}
\notag Vol(V(x; X)) \leq \kappa_d R(x; X)^d, 
\end{align}
\noindent where $\kappa_d = \pi^{d/2} / \Gamma(d/2 + 1)$ is the volume of the unit ball in $\mathbb{R}^d$.

\begin{lem}\label{lem:bound_ukq} Assume there exist $S_+(K), \alpha > 0$, such that $Vol(\partial K^r) \leq S_+(K) r^{\alpha}$ for all $r > 0$. Let
\begin{align}
\notag U_k(i) = \mathbf{1}_{d(X_i, \partial K) \leq R_k(X_i; X)} R_k(X_i; X)^d.
\end{align}
Then, for some $c_{d, qd + \alpha, k} >0$, 
\begin{align}
\notag \mathbb{E} U_k^q(i) \leq S_+ (K) c_{d,q d + \alpha, k} n^{- q - \alpha/d}, 
\end{align}
\noindent for all $n \geq 1$, $q \geq 1$.
\end{lem}

\begin{proof} To simplify computations,  introduce the process $X'$ defined as
\begin{align}
\notag X' = \bigcup_{m \in \mathbb{Z}^d} (X + m).
\end{align}
\noindent Unlike the independent setting in~\cite{LRP}, here the law of $X'$ is only invariant under integer valued translations. Note that a.s.~$X'$ has exactly $n$ points in any cube $[t, t+1]^d$, where $t \in \mathbb{R}$. Let $T_x = \{ [y, y+1]^d : y \in \mathbb{R}^d, x \in [y, y+1]^d\}$. Define $\overline{R_k}(x; X)$ as
\begin{align}
\notag \overline{R_k}(x; X)  \coloneqq \sup_{T \in T_x} R_k (x; X' \cap T). 
\end{align}

\noindent Note that if $x \in [0,1]^d$, then $[0,1 ]^d \in T_x$ and so $\overline{R_k}(x; X') \geq R_k(x; X)$. When the $X_i$ are sampled independently and uniformly, as in~\cite{LRP}, it is the case that $\overline{R_k}(x; X')$ does not depend on the position of $x$. However, in the hidden Markov model case we need to find a further bound on $\overline{R_k}(x; X')$.

\noindent For that purpose, consider the cube $K_0 \coloneqq [-1/2, 1/2]^d$ of volume $1$ centered at $\mathbf{0} \in \mathbb{R}^d$. Let $B_A$ be the open ball of $\mathbb{R}^d$, centered at $\mathbf{0}$, and of volume $A < 1$, to be chosen later. Next,  let $\tilde{X} = (0, \tilde{X}_1, \ldots, \tilde{X}_{n-1})$ be such that $\tilde{X}_i \in K_0$, for all $i = 1, \ldots, n-1$. Furthermore, for any Lebesgue measurable $T \subseteq K_0$, set
\begin{align}
\notag \mathbb{P}(\tilde{X}_i \in T) = c_m  |T \cap B_A| + c_M |T \cap B_A^c|,
\end{align} 
\noindent for all $i \in 1, \ldots, n-1$ where $|\cdot|$ now denotes the Lebesgue measure of the corresponding sets. If $A = (c_M - 1)/(c_M - c_m)$, then the above is a well-defined positive measure on $K_0$. From the restrictions of the hidden Markov model, if $\tilde{R}_k = R_k(0; \tilde{X})$, 
\begin{align}
\notag \overline{R_k}(x; X) \leq \tilde{R}_k.
\end{align}
\noindent Indeed, $\tilde{R}_k$ represents the worst-case scenario where the remaining points of $X$ are least likely to be distributed in the volume closest to $x$.

\noindent Then, 
\begin{align}
\label{eq:ukq_splus}  \mathbb{E} U_k^q(i) \leq  \mathbb{E}_{X_i, \tilde{X}}[ \mathbf{1}_{d(X_i; \partial K) \leq \tilde{R}_k} \tilde{R}_k^{qd}] \leq   S_+(K) \mathbb{E}_{\tilde{X}}[\tilde{R}_k^{qd+ \alpha}],
\end{align}
\noindent where we have used the upper bound on $Vol(\partial K^r)$.

\noindent To estimate $\mathbb{E}[ \tilde{R}_k^{qd+ \alpha}]$, note that if $\tilde{R}_k \geq r$, there will be a open ball of radius $r/2k$ in $K_0$ containing no points of $\tilde{X}$. Moreover, there will be $s_d \in (0,1)$, depending only on the dimension $d$, such that every ball of radius $2k$ contains a cube of side length $s_d r/k$ of the form $[g - s_d r/ 2k, g + s_d r/ 2k]$ where $g \in (s_d r / k) \mathbb{Z}^d $. Then, if $s_d r / k < 1$, 
\begin{align}
\notag \mathbb{P} (\tilde{R}_k \geq r) \leq & \mathbb{P} ( \exists g \in (s_d r / k ) \mathbb{Z}^d : \tilde{X} \cap [g - s_d r/ 2k, g + s_d r/ 2k] = \mathbf{0}) \\
\notag \leq & \#  \{ g  :   g \in (s_d r / k) \mathbb{Z}^d \cap [-r, r]^d\} \mathbb{P}( \tilde{X} \cap [ - s_d r/ 2k,  s_d r/ 2k] = \mathbf{0}) \\
\notag \leq & \frac{k^d}{(s_d )^d} (1 - c_m (s_d r / k)^d)^{n-1}.
\end{align}
If, on the other hand, $s_d r / k \geq 1$, $ \tilde{X} \cap [g - s_d r/ 2k, g + s_d r/ 2k]  = \tilde{X}$ and $ \mathbb{P}(\tilde{R}_k \geq r) = 0$. Then, using $1 - x \leq e^{-x}$, for any $u > 0$,  
\begin{align}
\notag \mathbb{E}[ \tilde{R}(0, \tilde{X})^u ] = & \int_0^{\infty} \mathbb{P}( \tilde{R}(0, \tilde{X}) \geq r^{1/u}) dr \\
\notag  \leq & c_{d, k} \int_0^{\infty} (1 - c_m (s_d r^{1/u} / k)^d)^{n-1} dr \\
\notag \leq & c_{d, k} \int_0^{\infty} \exp( - c_m (n-1) (s_d r^{1/u} / k)^{d}) dr \\
\notag \leq & c_{d, k, u} (n-1)^{u/d} \int_0^{\infty} \exp (-r^{d/u}) dr.
\end{align}

\noindent Applying the above in~\eqref{eq:ukq_splus} yields
\begin{align}
\notag  \mathbb{E} U_k^q(i) \leq c_{d, k, qd + \alpha} S_+(K) n^{- q-  \alpha/ d},
\end{align}
\noindent where $c_{d, k, qd + \alpha} > 0$ depends only on the parameters of the transition probabilities of the hidden chain and on $d, k$ and $qd + \alpha$, but neither on $n$ nor on $i$.

\end{proof}

\noindent Again, as before, we introduce a set of instructions $R$ and a function $h$, such that $h(R) = \varphi(X)$. We apply Proposition~\ref{thm:wass_hmm} and the initial step is to bound $\mathbb{E}[|\Delta_i h(R)|^r]$, where $ r > 0$. 

\noindent Let $S(R)$ be the original set of points generated by $R$ and $S(R^i)$ be the set of points generated after the change in the instruction $R_i$. The following proposition is the version of~\cite[Proposition 6.4]{LRP} for our framework.

\begin{prop}\label{prop:conds_voronoi} (i) If for every $s \in S(R) \setminus S(R^i)$, the set $R_1(s, S(R))$, containing $s$ and all its neighbors, is either entirely in $K$, or entirely in $K^c$, then $\Delta_i h(R)= 0$. A similar result holds for $s \in S(R^i) \setminus S(R)$ and the set $R_1(s, S(R^i))$.

\noindent (ii) Assume $|i - j|$ is large enough, so that $(S(R^i) \setminus S(R) ) \cup (S(R^j) \setminus S(R)) = S(R^{ij}) \setminus S(R)$,where $S(R^{ij})$ is the set of points generated after the changes in both $R_i$ and $R_j$. If for every $s_1 \in S(R^i) \Delta S(R)$ and $s_2 \in S(R^j) \Delta S(R)$, at least one of the following holds: 
\begin{enumerate}
\item $d_V( s_1, s_2; S(R^{ij}) \cap S(R)) \geq 2$, or
\item $d_V(s_1, \partial K; S(R^{ij}) \cap S(R)) \geq 2$ and $d_V(s_2, \partial K; S(R^{ij}) \cap S(R) ) \cap S(R)) \geq 2$, 
\end{enumerate}
\noindent then $\Delta_{i,j} h(R) = 0$.

\end{prop}

\noindent In similarity to the proof of Theorem~\ref{thm:cov}, then write
\begin{align}
\notag |\Delta_i h(R)| \leq &  \sum_{s \in S(R) \setminus S(R^i)} \mathbf{1}_{d_{S(R)}(s, \partial K) \leq R_1(s; S(R))} k_d R_1(s; S(R))^d  \\
\notag & +  \sum_{s \in S(R^i) \setminus S(R)} \mathbf{1}_{d_{S(R^i)}(s, \partial K) \leq R_1(s; S(R^i))} k_d R_1(s; S(R^i))^d. 
\end{align}

\noindent As before for some $T > 0$, there is an event $E$ and $\ep > 0$, such that conditioned on $E$, $|S(R^i) \setminus S(R)| = |S(R) \setminus S(R^i)| \leq T$ and $\mathbb{P}(E^c) \leq (1 - \ep)^T$. Then, from Lemma~\ref{lem:bound_ukq} there is $S_+(K), \alpha > 0$, such that  
\begin{align}
\notag \mathbb{E} |\Delta_i h(R)|^r \leq c_{d,r, \alpha}   (1 - \ep)^T + c_{d,r, \alpha} S_+(K)  T^r  n^{- r - \alpha /d}, 
\end{align} 
\noindent where $c_{d,r , \alpha}$ depends on the parameters of the model, the dimension $d$, as well as $r$ and $\alpha$. If $T = c\ln n$, for a suitable $c > 0$, then
\begin{align}
\label{eq:vor_delta_r} \mathbb{E} |\Delta_i h(R)|^r  \leq c_{d,r, \alpha} S_+(K)  (\ln n)^r  n^{- r - \alpha /d}.
\end{align}
\noindent An application of the Efron-Stein's inequality then yields~\eqref{eq:vor_mom}. Moreover, for the non-variance term in Theorem~\ref{thm:wass_hmm}, we have
\begin{align}
\label{eq:noncov_vor} \frac{1}{4 \sigma^3} \sum_{j = 0}^{|R| - 1} \sqrt{\mathbb{E} | \Delta_j h(R)|^6} + \frac{\sqrt{2 \pi} } {16 \sigma^3} \sum_{j = 0}^{|R| - 1} \mathbb{E} |\Delta_j h(R)|^3 \leq C \sigma^{-3}   (\ln n)^3  n^{- 3 - \alpha /d}.
\end{align}  

\noindent Next we analyze 
\begin{align}
\label{eq:brh_voronoi} B_{|R|} (h) \coloneqq \sup_{Y, Y', Z, Z'} \mathbb{E}[ \mathbf{1}_{\Delta_{i,j} h(Y) \neq 0} \mathbf{1}_{\Delta_{j,k} h(Y') \neq 0} |\Delta_j h(Z)|^2 |\Delta_k h(Z')|^2 ],
\end{align}
\noindent where as before the supremum is taken over recombinations  $Y, Y', Z, Z'$ of $R, R', \tilde{R}$. Let $E$ be the event that all perturbations of the instructions in~\eqref{eq:brh_voronoi} propagate at most $T$ levels. There is $\ep > 0$, depending only on the parameters of the models, such that $\mathbb{P}(E^c) \leq (1 - \ep)^T$.

\noindent As before, conditioned on $E$, if $|j - i| \leq |R| K$, the perturbation in $i$ might be propagating past the position, corresponding to instruction $j$, leading to difficulties in the analysis of $\Delta_{i,j} h(Y)$. This is the reason for conditioning further on the events $E_1, E_2, E_3$ corresponding to respectively $0, 1,$ or $2$ of the conditions $\{|i - j| \geq |R|K, |j - k| \geq |R|K\}$ holding. Note that $E_1, E_2$ and $E_3$ are deterministic. 
 
\noindent In this setting, we also study the event that all Voronoi cells are small. For that purpose, as in~\cite{LRP}, introduce the event $\Omega_n(X)$,
\begin{align}
\notag \Omega_n (X) \coloneqq \left( \max_{1 \leq j \leq n}  R(X_j; X) \leq n^{-1/d} \rho_n \right), 
\end{align}
\noindent where $\rho_n = (\ln n)^{1/d + \ep'}$ for $\ep'$ sufficiently small. Then, after conditioning on the realization of the hidden chain, a proof as in~\cite[Lemma 6.8]{LRP} leads to
\begin{align}
\label{eq:omegan} n^{\eta} ( 1 - \mathbb{P}(\Omega_n (X))) \to 0,
\end{align} as $n \to \infty$, and for all $\eta > 0$. 

\noindent We now estimate $B_{|R|}(h)$. Write, 
\begin{flalign}
\notag  \mathbb{E} & [ \mathbf{1}_{\Delta_{i,j} h(Y) \neq 0} \mathbf{1}_{\Delta_{j,k} h(Y') \neq 0} |\Delta_j h(Z)|^2 |\Delta_k h(Z')|^2 ] \\
\notag = & \quad \mathbb{E}  [ \mathbf{1}_{\Delta_{i,j} h(Y) \neq 0} \mathbf{1}_{\Delta_{j,k} h(Y') \neq 0} |\Delta_j h(Z)|^2 |\Delta_k h(Z')|^2 \mathbf{1}_{E^c}] \\
\notag & \quad+ \mathbb{E}  [ \mathbf{1}_{\Delta_{i,j} h(Y) \neq 0} \mathbf{1}_{\Delta_{j,k} h(Y') \neq 0} |\Delta_j h(Z)|^2 |\Delta_k h(Z')|^2 \mathbf{1}_{E} \mathbf{1}_{\Omega_n^c}]  \\
\notag & \quad+ \mathbb{E}  [ \mathbf{1}_{\Delta_{i,j} h(Y) \neq 0} \mathbf{1}_{\Delta_{j,k} h(Y') \neq 0} |\Delta_j h(Z)|^2 |\Delta_k h(Z')|^2 \mathbf{1}_{E} \mathbf{1}_{\Omega_n} \mathbf{1}_{E_1}]\\
\notag &\quad + \mathbb{E}  [ \mathbf{1}_{\Delta_{i,j} h(Y) \neq 0} \mathbf{1}_{\Delta_{j,k} h(Y') \neq 0} |\Delta_j h(Z)|^2 |\Delta_k h(Z')|^2 \mathbf{1}_{E} \mathbf{1}_{\Omega_n} \mathbf{1}_{E_2}]\\
\label{eq:brh_vor_cond} & \quad+ \mathbb{E}  [ \mathbf{1}_{\Delta_{i,j} h(Y) \neq 0} \mathbf{1}_{\Delta_{j,k} h(Y') \neq 0} |\Delta_j h(Z)|^2 |\Delta_k h(Z')|^2 \mathbf{1}_{E} \mathbf{1}_{\Omega_n} \mathbf{1}_{E_3}].
\end{flalign} 

\noindent Using $|\Delta_j h(Z)|, |\Delta_k h(Z')| \leq 1$, we get that the first two terms in~\eqref{eq:brh_vor_cond} are bounded by $\mathbb{P}(E^c)  +\mathbb{P}(\Omega_n^c)$. Next, 
\begin{align}
\notag \mathbb{E} &  [ \mathbf{1}_{\Delta_{i,j} h(Y) \neq 0} \mathbf{1}_{\Delta_{j,k} h(Y') \neq 0} |\Delta_j h(Z)|^2 |\Delta_k h(Z')|^2 \mathbf{1}_{E} \mathbf{1}_{\Omega_n} \mathbf{1}_{E_1}] && \\
\notag \leq & \quad  \mathbf{1}_{E_1}  \mathbb{E} [  |\Delta_j h(Z)|^2 |\Delta_k h(Z')|^2 \mathbf{1}_{E} \mathbf{1}_{\Omega_n}] &&\\
\label{eq:brh_vor_e1} \leq & \quad C \mathbf{1}_{E_1} T^4 n^{-4 - 2 \alpha/d} \rho_n^{4d}, &&
\end{align} 
\noindent where we have used the Cauchy-Schwarz inequality. 

\noindent Next, define as before, 
\begin{align}
\notag  S_1 \coloneqq & S(Y) \Delta S(Y^i), \quad S_2 \coloneqq  S(Y)  \Delta S(Y^j),  \\
\notag S_3 \coloneqq & S(Y') \Delta S((Y')^j), \quad S_4 \coloneqq  S(Y') \Delta S((Y')^k). 
\end{align}
\noindent Further, let $S_0 = S(Y) \cap S(Y^i) \cap S(Y^j)$ and $S_0' =  S(Y') \cap S((Y')^j) \cap S((Y')^k)$. By Proposition~\ref{prop:conds_voronoi}(ii), it follows that conditioned on $\Omega_n$, 
\begin{align}
\notag \mathbf{1}_{\Delta_{i,j} h(Y) \neq 0} \leq & \sum_{s_1 \in S_1, s_2 \in S_2} \mathbf{1}_{d_{S_0}(s_1, \partial K) \leq 2n^{-1/d} \rho_n }  \mathbf{1}_{d_{S_0}(s_2, \partial K) \leq 2n^{-1/d} \rho_n } \mathbf{1}_{d_{S_0}(s_1, s_2) \leq 2n^{-1/d} \rho_n}.
\end{align}
\noindent Conditioned on $E_3$, the sets $S_1, S_2 \cup S_3$ and $S_4$ are pairwise disjoint. 
\begin{align}
\notag \mathbb{E} &  [ \mathbf{1}_{\Delta_{i,j} h(Y) \neq 0}   \mathbf{1}_{\Delta_{j,k} h(Y') \neq 0} |\Delta_j h(Z)|^2 |\Delta_k h(Z')|^2 \mathbf{1}_{E} \mathbf{1}_{\Omega_n} \mathbf{1}_{E_3}] &&\\
\notag \leq & \quad C \mathbf{1}_{E_3} T^4 n^{-4 - 2 \alpha/d} \rho_n^{4d} \mathbb{E} [ \mathbf{1}_{\Delta_{i,j} h(Y) \neq 0}  \mathbf{1}_{\Delta_{j,k} h(Y') \neq 0}   \mathbf{1}_{E} \mathbf{1}_{\Omega_n}].&&
\end{align} 
\noindent By conditioning on the realization of all  hidden chains $H$.
\begin{align}
\notag  \mathbb{E} &  [ \mathbf{1}_{\Delta_{i,j} h(Y) \neq 0}  \mathbf{1}_{\Delta_{j,k} h(Y') \neq 0}   \mathbf{1}_{E} \mathbf{1}_{\Omega_n}]  \\
\notag = & \mathbb{E}[ \mathbb{E} [ \mathbf{1}_{\Delta_{i,j} h(Y) \neq 0}  \mathbf{1}_{\Delta_{j,k} h(Y') \neq 0}   \mathbf{1}_{E} \mathbf{1}_{\Omega_n} | H]]  \\
\notag \leq & \mathbb{E} \bigg[  \mathbb{E}\bigg[ \sum_{ \substack{ s_1 \in S_1, s_2 \in S_2 \\ s_1' \in S_3, s_2' \in S_4}}   \mathbf{1}_{d_{S_0}(s_1', \partial K) \leq 2n^{-1/d} \rho_n } \mathbf{1}_{d_{S_0}(s_1, s_2) \leq 2n^{-1/d} \rho_n} \mathbf{1}_{d_{S_0'}(s_1', s_2') \leq 2n^{-1/d} \rho_n}  \mathbf{1}_{E} \mathbf{1}_{\Omega_n} | H \bigg] \bigg] \\
\notag \leq & \mathbb{E} \mathbb{E} \bigg[ \sum_{ s_2 \in S_2, s_1' \in S_1} \mathbf{1}_{d_{S_0}(s_1', \partial K) \leq 2n^{-1/d} \rho_n }  \mathbf{1}_{E} \mathbf{1}_{\Omega_n}  \\ 
\notag & \quad \quad \mathbb{E} \bigg[   \sum_{ s_1 \in S_1, s_2' \in S_4}  \mathbf{1}_{d_{S_0}(s_1, s_2) \leq 2n^{-1/d} \rho_n}   \mathbf{1}_{d_{S_0'}(s_1', s_2') \leq 2n^{-1/d} \rho_n} \bigg| s_1', s_2\bigg]  \bigg|  H \bigg]. 
\end{align}
\noindent Now, conditioned on $H$, $s_1'$ and $s_2$, we have independence in the innermost expectation. Therefore, the above is bounded by
\begin{align}
\notag  \mathbb{E}    \bigg[ \sum_{ s_2 \in S_2, s_1' \in S_1} \mathbf{1}_{d_{S_0}(s_1', \partial K) \leq 2n^{-1/d} \rho_n }  \mathbf{1}_{E} \mathbf{1}_{\Omega_n} 4 T^2 2^d n^{-2} \rho_n^{2d}  \bigg] \leq C T^4 n^{-2} \rho_n^{2d} n^{- \alpha /d} \rho_n^{\alpha}. 
\end{align} 
\noindent Then, 
\begin{align}
\notag \mathbb{E} &  [ \mathbf{1}_{\Delta_{i,j} h(Y) \neq 0}  \mathbf{1}_{\Delta_{j,k} h(Y') \neq 0} |\Delta_j h(Z)|^2 |\Delta_k h(Z')|^2 \mathbf{1}_{E} \mathbf{1}_{\Omega_n} \mathbf{1}_{E_3}] && \\
\label{eq:brh_vor_e3} \leq & \quad  C \mathbf{1}_{E_3} T^8 n^{-6 - 3\alpha/d} \rho_n^{6d + \alpha}. &&
\end{align}
\noindent Finally, for the event $E_2$, assuming that $|i - j | \geq |R| K$, the other case being identical, 
\begin{align}
\notag \mathbb{E} &  [ \mathbf{1}_{\Delta_{i,j} h(Y) \neq 0}  \mathbf{1}_{\Delta_{j,k} h(Y') \neq 0} |\Delta_j h(Z)|^2 |\Delta_k h(Z')|^2 \mathbf{1}_{E} \mathbf{1}_{\Omega_n} \mathbf{1}_{E_2}]&& \\
\notag \leq &  \quad \mathbb{E} [ \mathbf{1}_{\Delta_{i,j} h(Y) \neq 0}   |\Delta_j h(Z)|^2 |\Delta_k h(Z')|^2 \mathbf{1}_{E} \mathbf{1}_{\Omega_n} \mathbf{1}_{E_2}]&& \\
\label{eq:brh_vor_e2} \leq & \quad  C \mathbf{1}_{E_2} T^6 n^{-5 - 3 \alpha/d} \rho_n^{5d + \alpha}.&&
\end{align}

\noindent Using~\eqref{eq:brh_vor_cond}, ~\eqref{eq:brh_vor_e1}, ~\eqref{eq:brh_vor_e2}, and~\eqref{eq:brh_vor_e3}, leads to
\begin{align}
\notag B_{|R|}(h) \leq & C ((1 - \ep)^T  +\mathbb{P}(\Omega_n^c) +   \mathbf{1}_{E_1} T^4 n^{-4 - 2\alpha/d} \rho_n^{4d} \\
\notag & \quad  +\mathbf{1}_{E_2} T^6 n^{-5 - 3\alpha/d} \rho_n^{5d + \alpha}  +  \mathbf{1}_{E_3} T^8 n^{-6 -3 \alpha/d} \rho_n^{6d + \alpha}).
\end{align}

\noindent Similar arguments yield, 
\begin{align}
\notag B_{|R|}^{(k)}(h) \leq C (\mathbf{1}_{E_1} T^4 n^{-4 - 2\alpha/d} \rho_n^{4d} +  \mathbf{1}_{E_2} T^6 n^{-5 - 3\alpha/d} \rho_n^{5d + \alpha}),\\
\notag B_{|R|}^{(j)}(h) \leq C ( \mathbf{1}_{E_1} T^4 n^{-4 - 2\alpha/d} \rho_n^{4d}  + \mathbf{1}_{E_2} T^6 n^{-5 - 3\alpha/d} \rho_n^{5d + \alpha}).
\end{align}
 
\noindent Then, 
\begin{align}
\notag  \sum_{i = 0}^{|R| - 1} & \sum_{j, k \notin A} \mathbf{1}_{i \neq j \neq k} B_{|R|} (h) \\ 
\notag \leq &  C ( n^3 ( 1 - \ep)^T + n^3 \mathbb{P}(\Omega_n^c) + T^6 n^{-3 - 2\alpha/d} \rho_n^{4d} +T^7 n^{-3 - 3\alpha/d} \rho_n^{5d + \alpha}  +   T^8 n^{-3 - 3\alpha/d} \rho_n^{6d + \alpha}) \\
\notag \leq & C ( n^{-3 - 2 \alpha/d} (\ln n)^{10 + 4 \ep'}),
\end{align}
\noindent where we have chosen $K = c \ln n$, for a suitable $c > 0$, independent of $n$, using also~\eqref{eq:omegan} and the definition of $\rho_n$.

\noindent Moreover,
\begin{align}
\notag  \sum_{i = 0}^{|R| -1} \sum_{j, k \notin A} (\mathbf{1}_{i \neq j = k} + \mathbf{1}_{i = k \neq j} ) B_{|R|}^{(k)}(h)   \leq C ( n^{-3 - 2 \alpha/d} (\ln n)^{10 + 4 \ep'}) , \\
\notag  \sum_{i = 0}^{|R| -1} \sum_{j, k \notin A}  (\mathbf{1}_{i \neq j = k} + \mathbf{1}_{i = j \neq k} ) B_{|R|}^{(j)}(h) \leq C ( n^{-3 - 2 \alpha/d} (\ln n)^{10 + 4 \ep'}).
\end{align}

\noindent The bounds on the variance terms in Proposition~\ref{prop:var_b} become
\begin{align}
\notag  \sqrt{Var(\mathbb{E}[U|R])} \leq & \frac{1}{\sqrt{2}} \sum_{A \subsetneq[|R|]} k_{|R|, A} \bigg(  C  ( n^{-3 - 2 \alpha/d} (\ln n)^{10 + 4 \ep'}) \bigg)^{1/2}\\
\label{eq:cov_var_vor} \leq  & C \sqrt{n} ( n^{ -2 - \alpha/d} (\ln n)^{5 + 2\ep'}).
\end{align}

\noindent Then,~\eqref{eq:thm_cov_vor} follows from~\eqref{eq:cov_var_vor},~\eqref{eq:noncov_vor} and Proposition~\ref{thm:wass_hmm}.
\end{proof}

\noindent Before the proof of the main result is presented, recall the following result~(\cite[Corollary 2.4]{LRP}) concerning the variance.  Let  $X \coloneqq (X_1, \ldots, X_n) \in E^n$, where $E$ is a Polish space. If $X'$ is an independent copy of $X$, and $f: E^n \to \mathbb{R}$ is measurable, with $\mathbb{E}[ f(X)^2] < \infty$,
\begin{align}
\label{eq:var_lower} Var(f(X)) \geq \sum_{i = 1}^n \mathbb{E}[ ( \mathbb{E} [ \Delta_i f(X',X) |X ] )^2].
\end{align} 

\noindent In our setting we take $f = \varphi$. Unlike~\cite{LRP}, the function $\varphi$ is not symmetric and right-hand side of~\eqref{eq:var_lower} cannot be simplified.  

\begin{proof}[Proof of Theorem~\ref{thm:vor_main}] Note that~\eqref{eq:vor_mom} was proved earlier via an application of Efron-Stein's inequality to~\eqref{eq:vor_delta_r}. Furthermore, ~\eqref{eq:vor_clt} follows from Theorem~\ref{thm:vor_clt} and~\eqref{eq:vor_var}. Thus, only~\eqref{eq:vor_var} is left to prove.
\noindent Let $H$ is the realization of the hidden chain for $X$. By the law of the total variance, $Var(\varphi(X)) \geq Var(\varphi(X) | H)$. Let $X'$ be an independent copy of $X$, given $H$. Note that, given $H$, $(X_i)_{i = 1, \ldots, n}$ and $(X_i')_{ i = 1, \ldots, n}$ are independent random variables which are \emph{not} identically distributed. 

\noindent Applying~\eqref{eq:var_lower} to $\varphi(X | H)$, it follows that
\begin{align}
\notag Var(\varphi(X) | H) \geq \sum_{i = 1}^n \mathbb{E}_{X_i'}^H ( \mathbb{E}_X^H [ \varphi(X^i) - \varphi(X) ] )^2,
\end{align}
\noindent where  $X^i = (X_1, \ldots, X_{i-1}, X_i', X_{i+1}, \ldots, X_n)$, and $\mathbb{E}^H$ signifies that $H$ is given. To simplify notation we drop the $H$. The difference with the proof in~\cite{LRP} is that now the variables are no longer identically distributed. Write
\begin{align}
\notag \mathbb{E}_X [ \varphi(X^i) - \varphi(X) ] = \mathbb{E}_X [ \varphi(X^i) - \varphi(X^{(i)}] - \mathbb{E}_X[ \varphi(X) - \varphi(X^{(i)})],  
\end{align}  
\noindent where $X^{(i)}  = (X_1, \ldots, X_{i-1},  X_{i+1}, \ldots, X_n)$.
\noindent By Lemma~\ref{lem:bound_ukq}, 
\begin{align}
\label{eq:var_part_det} \mathbb{E}_X[ \varphi(X) - \varphi(X^{(i)})] \leq c_{d, \alpha} n^{-1 - \alpha/d}.
\end{align}
\noindent We are left to study   $\mathbb{E} [ \varphi(X^i) - \varphi(X^{(i)}]$. Recall that
\begin{align}
\notag \varphi(X^i) - \varphi(X^{(i)}) = & \mathbf{1}_{\{X_i' \in K\}} \sum_{j \neq i} \mathbf{1}_{\{X_j \in K^C\}} v(X_i', X_j; X^{(i,j)}) \\
\notag &  - \mathbf{1}_{\{X_i' \in K^C\}} \sum_{j \neq i} \mathbf{1}_{\{X_j \in K\}} v(X_i', X_j; X^{(i,j)}),   
\end{align}

\noindent Now, for the case $X_i' \in K^C$ (the other case being equivalent). 
\begin{align}
\notag  |\mathbb{E}_{X, X_i'} & [  - \mathbf{1}_{\{X_i' \in K^C\}} \sum_{j \neq i} \mathbf{1}_{\{X_j \in K\}} v(X_i', X_j; X^{(i,j)}) ] | \\
\notag \geq & \quad \mathbb{E}_{X_i'} [\mathbf{1}_{\{X_i' \in \partial K_+^{n^{-1/d}} \}} \sum_{j \neq i} \mathbb{E}_X[  \mathbf{1}_{\{X_j \in K\}} v(X_i', X_j; X^{(i,j)}) ] ],
\end{align} 
\noindent since $v(X_i', X_j; X^{(i,j)}) \geq 0$. Then, 
\begin{align}
\notag  \mathbb{E}_X & [  \mathbf{1}_{\{X_j \in K\}} v(x, X_j; X^{(i,j)}) ] ]  \\
\notag \geq & \mathbb{E}_{X^{(i,j)}} [ c_1 \int_{y \in K} v(x, y; X^{(i,j)}) dy ] \\
\notag \geq &   c_1 Vol(B(x, \beta n^{-1/d}) \cap K) \inf_{y : ||x- y || \leq \beta n^{-1/d}} \mathbb{E}_{X^{(i,j)}}[  v(x,y; X^{(i,j)})],
\end{align} 
\noindent using the independence after conditioning on $H$ and the properties of the model. We want to find an event that implies that $v(x,y; X^{(i,j)}) \geq c n^{-1}$. One instance is when  no point of $X^{(i,j)}$ falls in $B(y, 6 \beta n^{-1/d})$. Indeed, then $B(y, 3 \beta n^{-1/d}) \subset V(y, X^{(i,j)})$. The distance between $y$ and $x$ is less than $\beta n^{-1/d}$, and so there is $z \in B(y, 3 \beta n^{-1/d})$, namely $z = x + \beta n^{-1/d} (x - y)/||x-y||$ such that
\begin{align}
\notag B(z, \beta n^{-1/d}) \subset V(x, (X^{(i,j)}, y)) \subset B(y, 3 \beta n^{-1/d}) \subset V(y; X^{(i,j)}).
\end{align}
\noindent Then, $v(x,y; X^{(i,j)}) \geq Vol(B(z, \beta n^{-1/d}) = \kappa_d \beta^d n^{-1}$. Finally, 
\begin{align}
\notag  \inf_{y : ||x- y || \leq \beta n^{-1/d}}  & \mathbb{E}_{X^{(i,j)}}[  v(x,y; X^{(i,j)})] \\
\notag \geq & \kappa_d \beta^d n^{-1} \mathbb{P} ( X^{(i,j)} \cap B(y, 6 \beta n^{-1/d}) = \emptyset) \\
\notag \geq & \kappa_d \beta^d n^{-1} ( 1 - c_2 \beta^d n^{-1})^n \\
\notag \geq & c_{d, \beta} n^{-1}, 
\end{align}
\noindent for some $c_{d, \beta} > 0$ depending on the parameters of the model, the dimension $d$ and $ \beta$. Then
\begin{align}
\notag  \mathbb{E}_X  [  \mathbf{1}_{\{X_j \in K\}} v(x, X_j; X^{(i,j)}) ] ]  \geq c_{d, \beta} Vol(B(x, \beta n^{-1/d}) ) n^{-1}.
\end{align}
\noindent Therefore, by the very definition of $\gamma(K, r, \beta)$ and since the case $X_i' \in K$ is symmetric, 
\begin{align}
\notag \mathbb{E}_{X_i'} \mathbb{E}_X [ (\varphi(X^i) - \varphi(X^{(i)})^2 ] \geq &  c_{d, \beta} \bigg( c_1 \int_{\partial K_+^{n^{-1/d}} } Vol(B ( x, \beta n^{-1/d}) \cap K)^2 dx \\ 
\notag & \quad + c_1 \int_{\partial K_-^{n^{-1/d}} } Vol(B ( x, \beta n^{-1/d}) \cap K^c)^2 dx\bigg)  \\
\notag = & c_{d, \beta} ( n^{-2} \gamma(K, n^{-1/d}, \beta) + n^{-2} \gamma(K^c, n^{-1/d}, \beta) ).
\end{align} 

\noindent If the rolling ball condition~\eqref{eq:rbc}, and the lower bound on $\partial{K}^{n^{-1/d}}$ both hold, then
\begin{align}
\notag \mathbb{E}_{X_i'} \mathbb{E}_X [ (\varphi(X^i) - \varphi(X^{(i)})^2 ]   \geq c_{d, \beta} S_-(K) \gamma(K, \beta) n^{-2 - \alpha/d},
\end{align} 
\noindent which dominates the contribution~\eqref{eq:var_part_det} from $\mathbb{E}[\varphi (X) - \varphi(X^{(i)})]$. Therefore, finally
\begin{align}
\notag Var(\varphi(X)) \geq c_{d, \beta}^- S_-(K) \gamma(K, \beta) n^{-1 - \alpha/d},
\end{align}
\noindent as desired. 
\end{proof}

\begin{rem}\label{rem:occupancy} Let us expand a bit on another potential application of our generic framework, namely the occupancy problem as studied in~\cite{gkp-2020}. To set up the notation, $(Z_1, \ldots, Z_n)$ is an aperiodic, irreducible and time homogeneous (hidden) Markov chain that transitions between different alphabets. Then to each alphabet is associated a distribution over the collection of all possible letters, giving rise to the observed letters $(X_1, \ldots, X_n)$. We assume that the number of alphabets is finite but that the number of total letters is $\lfloor\alpha n\rfloor$, for some fixed $\alpha >0$. One studies $W \coloneqq f(X_1, \ldots, X_n)$- the number of letters that have not appeared among the $X_1, \ldots, X_n$. Then, an analysis as in the proof of Theorem~\ref{thm:cov} leads to:
\begin{align}
\notag d_K\left(\frac{W - \mathbb{E} W}{\sqrt{Var (W)}} , \mathcal{N} \right) \leq C \left( \frac{ n(\ln n)^3}  {\sqrt{Var(W)^3}} +   \frac{ n^{1/2} (\ln n)^4}{ Var(W)} \right), 
\end{align}
where $Var(W)$ is a function of $n$, $\mathcal{N}$ is the standard normal distribution and $C > 0$ is a constant depending on the parameters of the model, but not on $n$. As mentioned at the beginning of the section, the study of the precise order of growth of the variance of $W$ is not within the scope of the current paper. For the iid case one can show, see, e.g.~\cite{englund-1981}, that $Var(W) \sim (\alpha e^{-1/ \alpha} - (1 + \alpha) e^{-2 / \alpha} ) n$, as $n \to \infty$. 
\end{rem}

\end{document}